\documentclass[a4paper,12pt]{article} 

\usepackage[boxed]{algorithm2e}
\usepackage{authblk}
\usepackage{blindtext}
\usepackage[utf8]{inputenc}
\usepackage[margin=2cm]{geometry}
\usepackage{enumerate}
\usepackage{amsmath,amsthm,amssymb,amsfonts}
\usepackage{todonotes}
\usepackage{graphicx}
\usepackage{stmaryrd}
\newtheorem{theorem}{Theorem}[section]
\newtheorem{lemma}[theorem]{Lemma}
\newtheorem{proposition}[theorem]{Proposition}
\newtheorem{corollary}[theorem]{Corollary}
\newtheorem{definition}[theorem]{Definition}
\newtheorem{remark}[theorem]{Remark}
\newtheorem{example}[theorem]{Example}
\newtheorem{assumption}[theorem]{Assumption}
\usepackage{multirow}
\usepackage{hyperref}
\hypersetup{
    colorlinks=true,
    linkcolor=blue,
    filecolor=magenta,      
    citecolor=blue,
    urlcolor=blue
    }

\usepackage{todonotes}

\DeclareMathOperator{\Id}{Id}
\DeclareMathOperator{\Fix}{Fix}
\DeclareMathOperator{\prox}{prox}
\DeclareMathOperator{\gra}{gra}
\DeclareMathOperator{\zer}{zer}
\DeclareMathOperator{\dom}{dom}

\DeclareMathOperator*{\argmin}{argmin}

\DeclareMathOperator{\dist}{dist}

\newcommand{\setto}{\rightrightarrows}

\title{Linear convergence of relocated fixed-point iterations}
\author[$^\ddagger$]{Felipe Atenas}
\author[*,$^\dagger$]{Farhana A. Simi}
\author[*]{Matthew K. Tam}
\affil[$^\ddagger$]{Centro de Modelamiento Matemático (CNRS IRL2807), Universidad de Chile, Santiago, Chile. Email: 
\href{mailto:fatenas@cmm.uchile.cl}{fatenas@cmm.uchile.cl}}
\affil[*]{School of Mathematics and Statistics, University of Melbourne, Parkville VIC 3010, Australia. Email: 
\href{mailto:fsimi@student.unimelb.edu.au}{fsimi@student.unimelb.edu.au}, \href{mailto:matthew.tam@unimelb.edu.au}{matthew.tam@unimelb.edu.au}
}
\affil[$^\dagger$]{Department of Mathematics, University of Dhaka, Bangladesh. 
Email:~\href{mailto:simi_96@du.ac.bd}{\texttt{simi\_96@du.ac.bd}}}

\begin{document}

\maketitle
\begin{abstract}
We establish linear convergence of relocated fixed-point iterations as introduced by Atenas et al.\ (2026) \href{https://doi.org/10.1137/25M1776810}{DOI: 10.1137/25M1776810} assuming the algorithmic operator satisfies a linear error bound. In particular, this framework applies to the setting where the algorithmic operator is a contraction. As a key application of our framework, we obtain linear convergence of the relocated Douglas--Rachford algorithm for finding a zero in the sum of two monotone operators in a setting with Lipschitz continuity and strong monotonicity assumptions. We also apply the framework to deduce linear convergence of variable stepsize resolvent splitting algorithms for multioperator monotone inclusions.
\end{abstract}


\section{Introduction}

Many algorithms in optimization can be formulated and analyzed as \emph{fixed-point iterations}; see, for example, \cite{krasnosel1955two,halpern1967fixed,fukushima1981generalized,polyak2007newton,rockafellar1976monotone,combettes2004solving}. Given an initial point $x_0$ in a real Hilbert space $\mathcal{H}$, a real parameter $\gamma\in\mathbb{R}_{++}$ and an operator $T_{\gamma}\colon\mathcal{H}\to\mathcal{H}$, its fixed-point iteration generates a sequence $(x_n)_{n\in\mathbb{N}}$ according to
\begin{equation}\label{eq:fixed-point-iteration}
x_{n+1} := T_{\gamma}x_n\quad\forall n\in\mathbb{N}. 
\end{equation}
Under suitable conditions on $T_\gamma$, the fixed-point sequence $(x_n)_{n\in\mathbb{N}}$  converges to a fixed-point of $T_\gamma$. For example, when $T_{\gamma}$ is  \emph{nonexpansive} and \emph{asymptotically regular}, Opial~\cite{MR211301} established convergence of $(x_n)_{n\in\mathbb{N}}$ in the weak topology.  In many situations, whether or not $T_\gamma$ satisfies these suitable conditions, performance depends on an appropriate choice of the parameter $\gamma$.

For example, consider the structured minimization problem
\begin{equation}\label{eq:f+g}
    \min_{x \in \mathcal{H}} f(x) + g(x),
\end{equation}
where $f: \mathcal{H} \to \mathbb{R}$ is smooth and $g: \mathcal{H} \to \mathbb{R} \cup \{+\infty\}$ is nonsmooth. The \emph{proximal gradient method} defines the sequence $(x_n)_{n\in\mathbb{N}}$ via \eqref{eq:fixed-point-iteration} with 
\begin{equation}\label{eq:T PG}
T_\gamma := \prox_{\gamma g}\circ\big( \Id - \gamma \nabla f),
\end{equation}
where $\prox_{\gamma g}\colon\mathcal{H}\to\mathcal{H}$ denotes the \emph{proximity operator} of $\gamma g$ defined by
$$ \prox_{\gamma g}:= \argmin_{y \in \mathcal{H}} \left\{ g(y) + \frac{1}{2\gamma}\|y- \cdot\|^2\right\}. $$
In the standard convex setting when $\nabla f$ is assumed to be $L$-Lipschitz continuous, the proximal gradient method converges to a fixed-point of \eqref{eq:T PG} so long as $\gamma \in (0,2/L)$, and the fixed-points of \eqref{eq:T PG} are the minimizers of $f+g$. When $L$ is not known, a suitable $\gamma$  can be computed via a linesearch. However, this does not directly fit within \eqref{eq:fixed-point-iteration} as this requires $\gamma$ to vary across iterations. For other settings where convergence is known, see \cite{lee2016gradient,jia2023convergence}. 

When $f$ and $g$ in \eqref{eq:f+g} are both nonsmooth, the \emph{Douglas--Rachford} method \cite{douglas1956numerical,lions1979splitting} can be used. This method can be described as the fixed-point iteration~\eqref{eq:fixed-point-iteration} given by the \emph{Douglas--Rachford operator} defined by
\begin{equation}\label{DR operator sec 1}
     T_{\gamma}:=\Id-\prox_{\gamma f}+\prox_{\gamma g}(2\prox_{\gamma f} - \Id).
 \end{equation} 
In the convex setting, the fixed-point sequence $(x_n)_{n\in\mathbb{N}}$ for the Douglas--Rachford method converges (weakly) to a point $x$ which is a fixed-point of the Douglas--Rachford operator~\cite{douglas1956numerical,svaiter2011weak}. However, in contrast to the proximal gradient method, the fixed-point $x$ does not solve the original problem, rather its \emph{shadow} $\prox_{\gamma f}(x)$ does. In fact, the fixed-point set of the Douglas--Rachford operator is known to be dependent on the stepsize $\gamma\in\mathbb{R}_{++}$, and for different parameters can even be disjoint \cite[Remark~4.13]{atenas2025relocated}. As a consequence, the standard approach to analyzing  fixed-point iterations (such as Opial~\cite{MR211301}) cannot be applied to the Douglas--Rachford method or its derivatives~\cite{davis2017convergence,rieger2020backward} when $\gamma$ is not constant. 

To address the issue of analyzing  fixed-point iterations with varying stepsize parameters, Atenas, Bauschke, Dao and Tam \cite{atenas2025relocated} introduced the concept of \emph{relocated fixed-point iterations}.  This framework is able to deal with fixed-point iterations defined by a parameterized family of operators. More precisely, given a set of parameters $\Gamma \subseteq \mathbb{R}_{++}$ and a family of operators $(T_{\gamma})_{\gamma \in \Gamma}$, a {relocated fixed-point iteration} is defined as \begin{equation*}
x_{n+1} = \mathcal{Q}_{\gamma_{n+1}\gets \gamma_n}T_{\gamma_n}x_n \quad \forall n\in\mathbb{N},
\end{equation*} where the operator $\mathcal{Q}_{\gamma_{n+1}\gets \gamma_n}\colon\mathcal{H}\to\mathcal{H}$ is called the \emph{fixed-point relocator} (see Definition~\ref{def: fixed-point relocator} below). For a convergent sequence $\gamma_n \to \gamma^*$ in $\Gamma$ that satisfies a summability condition, the authors show convergence of relocated fixed-point iterations to a point in the fixed-point set of $T_{\gamma^*}$ 
(see Theorem~\ref{thm:Convergence of fixed-point relocator} below). Importantly, this framework still applies to settings where the operators $(T_{\gamma_n})_{n\in\mathbb{N}}$ have no common fixed point, such as the Douglas--Rachford method. For the latter, it preserves the classical fixed-point step defined by the Douglas--Rachford operator \eqref{DR operator sec 1}, introducing variable stepsizes through fixed-point relocators. 

The relocated fixed-point iterations framework is an analytical, unifying tool for parametrized fixed-point iterations. It has been applied to obtain variable stepsize variants of (distributed) resolvent splitting methods, and extended in \cite{atenas2026variable} for (distributed) forward-backward methods. This framework provides explicit conditions that the sequence of variable stepsizes should satisfy to preserve the convergence guarantees of the constant stepsize counterparts, see, e.g., \cite[Section 5]{atenas2025relocated} and \cite[Sections 4 \& 5]{atenas2026variable}. In this manner, relocated fixed-point iterations constitute the theoretical base for the design of parameter tuning strategies that can enhance the numerical performance of methods that can be written as fixed-point iterations.

Briefly, let us mention other approaches in the literature that consider fixed-point iterations with varying operators. Given a family of operators $(T_n)_{n \in \mathbb{N}}$, it is common to consider a fixed-point iteration of the form
$$x_{n+1} = T_{n}x_n \quad\forall n \in \mathbb{N}. $$
In this context, Maul{\'e}n, Fierro and Peypouquet \cite{maulen2024inertial} examine inertial Krasnoselskii-Mann iterations with variable parameters, and Cegielski, Reich and Zalas \cite{cegielski2018regular} investigate fixed-point iterations under regularity assumptions on the family of operators. In both settings, it is assumed that the operators from $(T_n)_{n \in \mathbb{N}}$ have a common fixed-point which, as discussed in the previous paragraph, need not hold when $T_n $ is the operator   $T_{\gamma_n}$ in \eqref{DR operator sec 1} \cite[Remark 4.13]{atenas2025relocated}.  In the limit, the sequence $(x_n)_{n \in \mathbb{N}}$ is shown to converge to a point $x \in \cap_{n \in \mathbb{N}} \Fix T_n$. 

Regarding splitting methods specifically, Liang, Fadili and Peyr\'e \cite{liang2017local} propose a non-stationary Douglas--Rachford algorithm for convex minimization problems. The non-stationary nature of the method consists of replacing the constant stepsize of the classical Douglas--Rachford algorithm with a varying stepsize sequence. A different non-stationary variant of the Douglas--Rachford algorithm is proposed by Lorenz and Tran-Dinh \cite{lorenz2019non}, derived from a different iteration scheme, in which the steps defining the Douglas--Rachford operator are themselves modified, and the resulting method is analyzed in the more general setting of maximally monotone operators.  The non-stationary Douglas--Rachford algorithm \cite[eq. (12)]{lorenz2019non}  is a particular case of the relocated Douglas--Rachford algorithm (Algorithm~\ref{a:reloc-DR}); see \cite[Remark~4.12]{atenas2025relocated}. 
Although the approaches \cite{liang2017local,lorenz2019non,atenas2025relocated} differ in construction, their convergence analyses rely on similar assumptions on the stepsize sequences; see Remark~\ref{contrast to NDR} and \cite[Remark 4.10(iii)]{atenas2025relocated}.

Beyond the non-stationary Douglas--Rachford algorithm, adaptive stepsize rules have also been developed for the Alternating Direction Method of Multipliers (ADMM) \cite{gabay1976dual,glowinski1975approximation}, including the self-adaptive strategy of He, Yang and Wang \cite{he2000alternating}, and an adaptive stepsize rule based on spectral estimation by Xu, Figueiredo, Yuan, Studer and Goldstein \cite{xu2017adaptive}. Furthermore, the authors of \cite{lorenz2019non} deduce variable stepsize rules for ADMM from their analysis of the non-stationary Douglas--Rachford algorithm. In contrast, \cite{he2000alternating,xu2017adaptive} do not translate the adaptive stepsize rules from ADMM to the Douglas--Rachford framework (or vice versa) for general maximally monotone operators. Furthermore, Pedregosa and Gidel \cite{pedregosa2018adaptive} propose a variable stepsize variant of a three operator extension of the Douglas--Rachford algorithm, based on a backtracking procedure to induce a descent condition. For other splitting methods with variable parameters via backtracking, we refer the reader to \cite{malitsky2018first,giselsson2016line}.

In this paper, we investigate the convergence rate of relocated fixed-point iterations. In particular, we show $R$-linear convergence of relocated fixed-point iterations when the operators $(T_{\gamma})_{\gamma \in \Gamma}$ satisfy a regularity condition called \emph{bounded linear regularity}. A notable aspect of our analysis is that the underlying iteration is not \emph{Fej\'er monotone}. We then apply this general framework to obtain linear convergence of variable stepsize resolvent splitting methods. When the stepsize is constant, Giselsson and Boyd \cite{giselsson2016linear} previously obtained linear rates of convergence for the Douglas--Rachford algorithm in the optimization case assuming one of the functions is strongly convex with Lipschitz continuous gradient. Our analysis covers this setting as a particular instance of Corollary~\ref{cor: DR error bound}. In \cite{simi2025linear}, Simi and Tam also utilize model assumptions of similar nature to characterize the rate of convergence for a multioperator extension of the Douglas--Rachford method. In the nonmonotone case, Atenas \cite{A25} uses an error bound akin to (bounded) linear regularity to establish rates of convergence of the Douglas--Rachford method for nonconvex optimization problems.  %, a property we here referred as \emph{bounded linear regularity}. 

The remainder of this paper is organized as follows. In Section~\ref{section:preliminaries}, we present some preliminary results used in our analysis. In Section~\ref{section: 3}, we show our main results concerning linear convergence of relocated fixed-point iterations under bounded linear regularity of the family of operators. In Sections~\ref{s:DR} and~\ref{s:MT}, we use our results to derive linear convergence of the relocated Douglas--Rachford and Malitsky--Tam algorithms, respectively. Finally, concluding remarks are presented in Section~\ref{s:conclusion}.

\section{Preliminaries} \label{section:preliminaries}

Throughout this work, $\mathcal{H}$ denotes a real Hilbert space with inner product $\langle\cdot,\cdot\rangle$ and induced norm $\|\cdot\|$. A \emph{set-valued} operator is denoted $A:\mathcal{H}\setto\mathcal{H}$ and maps each point $x\in\mathcal{H}$ to a set $Ax\subseteq\mathcal{H}$. When $A$ is single-valued (\emph{i.e.,} $Ax$ is a singleton for all $x\in\mathcal{H})$, we write $A:\mathcal{H}\rightarrow\mathcal{H}$. The \emph{domain},  \emph{graph},  set of \emph{fixed-points} and set of \emph{zeros} of an operator $A\colon\mathcal{H}\setto\mathcal{H}$ are denoted $\dom A:=\{x\in\mathcal{H}:Ax\neq\emptyset\}$, 
$\gra A:=\{(x,u)\in \mathcal{H}\times\mathcal{H}:u\in Ax\}, \Fix A:=\{x\in \mathcal{H}:x\in Ax\}$, and $\zer A:=\{x\in \mathcal{H}:0\in Ax\}$, respectively. The \emph{identity operator} is denoted $\Id:\mathcal{H}\rightarrow \mathcal{H}$. The inverse of an operator $A:\mathcal{H}\setto\mathcal{H}$ is denoted $A^{-1}:\mathcal{H}\setto\mathcal{H}$, whose graph is defined as $\gra A^{-1}:=\{(y,x)\in\mathcal{H}\times\mathcal{H}:y\in Ax\}$. Observe that a direct calculation assures that for any $\gamma \in \mathbb{R}_{++}$, \begin{equation} \label{eq:escaled-inverse}
    (\gamma A)^{-1} = A^{-1}\circ (\gamma^{-1} \Id).
\end{equation}

An operator $A:\mathcal{H}\setto\mathcal{H}$ is said to be \emph{monotone} if 
$$\langle x-y,u-v\rangle\geq0\quad\forall(x,u),(y,v)\in\gra A,$$
and it is \emph{maximally monotone} if there exists no monotone operator whose graph properly contains $\gra A$.
 It is said to be \emph{$\mu$-strongly monotone} if $\mu>0$ and $$\langle x-y,u-v\rangle\geq\mu\|x-y\|^2\quad\forall(x,u),(y,v)\in\gra A.$$

 The \emph{resolvent} of an operator $A:\mathcal{H}\setto\mathcal{H}$ is defined by $J_{A}:=(\Id+A)^{-1}$. When $A:\mathcal{H}\setto\mathcal{H}$ is a maximally monotone operator, the resolvent $J_{A}$ is single-valued with full domain and $(1/2)$-averaged~\cite[Corollary~23.11]{bauschke2017convex}. 

\begin{remark}\label{remark: Lipschitz continuity in resolvent}
    Let $\gamma\in\mathbb{R}_{++}$. Since $(x,\gamma) \mapsto J_{\gamma A}x$ is Lipschitz continuous on compact subsets of $\mathcal{H} \times \mathbb{R}_{++}$~\cite[Proposition~3.4]{atenas2025relocated}, then it is also (sequentially) continuous. Indeed, let $(x,\gamma) \in \mathcal{H} \times \Gamma$, and $(x_n, \gamma_n) \to (x,\gamma)$. Then the set $K = \{(x_n,\gamma_n)\}_{k \in \mathbb{N}} \cup \{(x,\gamma)\}$ is compact \cite[Theorem 2.38]{aliprantis2006infinite}, and thus there exists $L >0$ such that \begin{equation*}
    \| J_{\gamma_n A}x_n - J_{\gamma A} x\| \leq L \|(x_n,\gamma_n) - (x,\gamma)\| \quad \forall n \in \mathbb{N}.
\end{equation*} Taking the limit as $n \to \infty$ yields $J_{\gamma_n A}x_n \to J_{\gamma A} x$. 
\end{remark}
 Given an ordered pair of operators $(A_1,A_2)$ on $\mathcal{H}$, the \emph{Attouch--Th\'era primal problem}~\cite{attouch1996general} is the inclusion given by
    \begin{equation}\label{primal problem}
        \text{find~}x\in\mathcal{H}\text{~s.t.~}0\in (A_1+A_2)x,
    \end{equation}
and the \emph{Attouch--Th\'era dual problem} is the inclusion given by
    \begin{equation}\label{dual problem}
        \text{find~}y\in\mathcal{H}\text{~s.t.~}0\in A_1^{-1}y-A_2^{-1}(-y).
    \end{equation}
Note that the problems \eqref{primal problem} and \eqref{dual problem} are equivalent in the following sense: if $x$ solves \eqref{primal problem}, then there exists $y\in A_1x$ that solves \eqref{dual problem}. Conversely, if $y$ solves \eqref{dual problem}, then there exists $x\in A_1^{-1}y$ that solves \eqref{primal problem}. For further details, see~\cite[Corollary 3.2]{attouch1996general}. Attouch-Th{\'e}ra duality recovers \emph{Fenchel duality} from convex analysis under appropriate constraint qualifications \cite[Theorem~4.2]{attouch1996general}. 

A monotone operator $A\colon\mathcal{H}\setto\mathcal{H}$ is said to be \emph{paramonotone} if
\begin{equation}\label{def of paramonotone}
(x,u),(y,v)\in\gra A \text{~and~} \langle x-y,u-v\rangle=0\implies (x,v),(y,u)\in\gra A.
\end{equation}
Paramonotonicity is an important notion in the study of interior point methods and duality theory \cite{iusem1998some,bauschke2014rectangularity}. In particular, the subdifferential of a convex function $f$ is always paramonotone \cite[Proposition~2.2]{iusem1998some}. If $A=\partial f$ and two points $(x,u),(u,v)\in\gra A$ satisfy $\langle x-y,u-v\rangle=0$, then paramonotonicity implies that the same supporting hyperplane touches the epigraph of $f$ at both $(x,f(x))$ and $(y,f(y))$. More generally, paramonotonicity can be viewed as intermediate property between monotonicity and strict monotonicity. For further details, we refer the reader to \cite{iusem1998some,bauschke2014rectangularity,censor1998interior}.

When $A_1$ and $A_2$ are paramonotone operators, and $T_\gamma$ is the Douglas--Rachford operator defined as
\begin{equation}\label{DR operator in pri}
    T_{\gamma}: = \Id-J_{\gamma A_1}+J_{\gamma A_2}(2J_{\gamma A_1}-\Id),
    \end{equation}
then the fixed-points of $T_\gamma$ can be expressed as the sum of primal and rescaled dual solutions, as shown in the following lemma (cf. \cite[Proposition~1]{pena2021linear}).

\begin{proposition}\label{p:T_gamma-decomp}
    Let $\gamma \in \mathbb{R}_{++}$ and $A_1,A_2 \colon\mathcal{H}\setto\mathcal{H}$ be maximally monotone. Consider the Attouch--Th\'era primal-dual pair of problems
    \begin{equation*}
    \begin{aligned}
        \text{find~}x &\in \mathcal{P}_\gamma := \zer(\gamma A_1+ \gamma A_2), \\
        \text{find~}y &\in  \mathcal{D}_\gamma := \zer((\gamma A_1)^{-1}-(\gamma A_2)^{-1}\circ (-\Id)).
    \end{aligned}
    \end{equation*}
    Denote $\mathcal{P} := \mathcal{P}_1$ and $\mathcal{D} := \mathcal{D}_1$. Then $\mathcal{P} = \mathcal{P}_\gamma$ and $\gamma\mathcal{D} = \mathcal{D}_\gamma$. If, in addition,  $T_\gamma$ is the Douglas--Rachford operator \eqref{DR operator in pri} and $A_1$ and $A_2$ are paramonotone, then \begin{equation} \label{eq:para-duality}
    \Fix T_\gamma = \mathcal{P} + \gamma\mathcal{D}. %\text{~where~} \mathcal{P} := \zer(A_1+A_2) \text{~and~} \mathcal{D} := \zer(A_1^{-1}-A_2^{-1}\circ (-\Id)).
\end{equation}
\end{proposition}
\begin{proof}
    First, from definition, it holds that $\mathcal{P}:=\zer(A_1+A_2) = \zer(\gamma A_1 + \gamma A_2)$. In view of \eqref{eq:escaled-inverse}, we therefore have $ y\in \mathcal{D}_\gamma\iff 0 \in A_1^{-1}(\gamma^{-1}y) - A_2^{-1}(-\gamma^{-1}y)\iff 0 \in A_1^{-1}w - A_2^{-1}(-w)$ and $y = \gamma w\iff y \in \gamma  \mathcal{D}$.

Next, observe that the operator $T_\gamma$ in \eqref{DR operator in pri} is associated with the Attouch--Th\'era primal problem $0 \in (\gamma A_1 + \gamma A_2)x$, and the Attouch--Th\'era dual problem $0 \in (\gamma A_1)^{-1}y - (\gamma A_2)^{-1}(-y)$. When $A_1$ and $A_2$ are paramonotone, then $\gamma A_1$ and $\gamma A_2$ are also paramonotone. Since $\gamma A_1$ and $\gamma A_2$ are paramonotone, and noting that $\mathcal{P}=\mathcal{P}_\gamma$ and $\gamma\mathcal{D}=\mathcal{D}_\gamma$, it follows from \cite[Corollary~5.5(iii)]{bauschke2012attouch} that $\Fix T_\gamma=\mathcal{P}_{\gamma}+\mathcal{D}_{\gamma}=\mathcal{P}+\gamma\mathcal{D}.$
\end{proof}

A single-valued operator $T:\mathcal{H}\rightarrow\mathcal{H}$ is said to be \emph{$\beta$-Lipschitz} if $\beta\geq0$ and $$\|Tx-Ty\|\leq\beta\|x-y\|\quad\forall x,y\in\mathcal{H}. $$
An operator is said to be a \emph{$\beta$-contraction} if it is $\beta$-Lipschitz with $\beta\in[0,1)$. It is said to be  \emph{nonexpansive} if it is $1$-Lipschitz, and  \emph{$\alpha$-averaged nonexpansive} if $\alpha\in(0,1)$, and
$$\|Tx-Ty\|^2+\frac{1-\alpha}{\alpha}\|(\Id-T)x-(\Id-T)y\|^2\leq\|x-y\|^2\quad\forall x,y\in\mathcal{H}.$$

Given a nonempty closed convex set $C \subseteq \mathcal{H}$, $\dist(\cdot,C)\colon\mathcal{H}\to\mathbb{R}$ denotes the \emph{distance to $C$} given by $\dist(x,C) := \inf_{y \in C} \|y-x\|$ for all $x \in \mathcal{H}$. Furthermore, $P_C\colon\mathcal{H}\to C$ denotes the \emph{projection operator onto $C$} which maps each point $x\in\mathcal{H}$ to $P_Cx$, the unique point in $C$ such that $\|x - P_Cx\| = \dist(x,C)$. It holds that $P_C = J_{N_C}$ (see, e.g., \cite[Example~23.4]{bauschke2017convex}), where $N_C : \mathcal{H}\setto\mathcal{H}$ is the \emph{normal cone to $C$}, defined as $N_C(x) = \{ z \in \mathcal{H}: \langle z , y -x  \rangle \leq 0 \,\forall y \in C \}$ for $x \in C$, and $N_C(x) = \emptyset$ otherwise.

In \cite[Section~3.1]{pena2021linear}, the authors introduce the notion of strong convexity of a function relative to a set. % Following \cite{pena2021linear}, 
In this work, we define the concept of strong monotonicity of an operator relative to a set. The latter reduces to the former when the operator is the subdifferential of a convex function.
\begin{definition}[Strongly monotone relative to a set] \label{d:relative-to-a-set}
    Let $A:\mathcal{H}\setto\mathcal{H}$ be a set-valued 
    operator and $X\subseteq\dom A$ be a nonempty closed convex set. The operator $A$ is said to be \emph{$\mu$-strongly monotone relative to $X$} if
    \begin{equation*}\label{eq:relative strong monotonicity}
\langle v - u, y - P_X(y) \rangle \ge \mu \|y - P_X(y)\|^2\quad \forall (y, v) \in \operatorname{gra} A,\, u \in (A\circ P_X)y.
\end{equation*} 
\end{definition}
Similar to the observation in \cite[p.~10]{pena2021linear}, if $A$ is $\mu$-strongly monotone, then it is $\mu$-strongly monotone relative to any nonempty closed convex set $X\subseteq\dom A$. And, if $A=K^*BK$ where $B$ is $\mu$-strongly monotone and $K^*$ is a bounded linear operator, then $A$ is $\mu(\sigma^+_{\min}(K))^2$-strongly monotone relative to $K^{-1}(X)\subseteq\dom A$ for any nonempty closed convex subset $X\subseteq\dom B$, where $\sigma_{\min}^+(K)$ denotes the smallest positive singular value of $K$; see \cite[Lemma~4]{pena2021linear}. The following result characterizes strong monotonicity relative to a set of operators of the form $\gamma A$ and its inverse $(\gamma A)^{-1}$.

\begin{proposition} \label{p:rel-to-a-set} Let $\gamma \in \mathbb{R}_{++}$ and $A:\mathcal{H}\setto\mathcal{H}$ be a set-valued operator. Suppose $X\subseteq\dom A$ and $Y\subseteq\dom A^{-1}$ are nonempty closed convex sets. Then the following assertions hold.
    \begin{enumerate}[(i)]
        \item \label{p:gamma-rel-mon} If $A$ is $\mu$-strongly monotone relative to  $X$, then $\gamma A$ is $\gamma\mu$-strongly monotone relative to  $X$.
        \item \label{p:inverse-rel-mon}If $A^{-1}$ is $\mu$-strongly monotone relative to $Y$, then $(\gamma A)^{-1}$ is $\gamma^{-1}\mu$-strongly monotone relative to  $\gamma Y$.
    \end{enumerate}
\end{proposition}

\begin{proof}
(\ref{p:gamma-rel-mon})~Let $(y,v)\in\gra(\gamma A)$ and $u \in ((\gamma A) \circ P_X)y$. Then $(y,\gamma^{-1}v)\in\gra A$ and $\gamma^{-1} u \in (A \circ P_X)y$. Since $A$ is $\mu$-strongly monotone relative to $X$, we have
\begin{equation*}
            \langle \gamma^{-1}v - \gamma^{-1}u, y - P_X(y) \rangle \ge \mu \|y - P_X(y)\|^2.
        \end{equation*}
        Multiplying both sides by $\gamma \in \mathbb{R}_{++}$ yields the claimed result.

(\ref{p:inverse-rel-mon})~Let $(y,v)\in\gra(\gamma A)^{-1}$ and $u \in ((\gamma A)^{-1} \circ P_{\gamma Y})y$. From \eqref{eq:escaled-inverse} and $P_{\gamma Y} = \gamma P_Y \circ (\gamma^{-1} \Id)$ \cite[Proposition~29.1]{bauschke2017convex}, it follows that $(\gamma^{-1}y,v)\in\gra A^{-1}$ and

$u\in (A^{-1}\circ P_Y)(\gamma^{-1}y)$. Since $A^{-1}$ is $\mu$-strongly monotone relative to $Y$, we have 
\begin{equation*}
            \langle v - u, \gamma^{-1}y - P_Y(\gamma^{-1}y) \rangle \ge \mu \|\gamma^{-1}y - P_Y(\gamma^{-1}y)\|^2
\end{equation*}
Multiplying both sides by $\gamma \in \mathbb{R}_{++}$ followed by using the identity $P_{\gamma Y} = \gamma P_Y \circ (\gamma^{-1} \Id)$ gives
\begin{equation*}
            \langle v - u, y - P_{\gamma Y}(y) \rangle \ge  \gamma^{-1}\mu \|y - P_{\gamma Y}(y)\|^2,
        \end{equation*}
which proves the claim.
\end{proof}

In the following, we recall the definition of fixed-point relocators for a family of operators. Fixed-point relocators play a crucial role in the convergence analysis introduced in \cite{atenas2025relocated}, as they map fixed-point sets to fixed-point sets when the stepsize parameter varies alongside iterations. More specifically, given $\delta,\gamma >0$, a fixed-point relocator is an operator $\mathcal{Q}_{\delta\gets\gamma}$ for an algorithmic operator $T_\gamma$ that, in particular,  satisfies:  for any $x \in \Fix T_\gamma$, it holds $\mathcal{Q}_{\delta\gets\gamma}x \in \Fix T_\delta$. Namely, $\mathcal{Q}_{\delta\gets\gamma}$ relocates fixed-points of $T_\gamma$ to fixed-points of $T_\delta$. Definition~\ref{def: fixed-point relocator} formalizes this idea, and also states continuity properties necessary for the convergence results presented in Theorem~\ref{thm:Convergence of fixed-point relocator}.
\begin{definition}[{Fixed-point relocator \cite[Definition~4.1]{atenas2025relocated}}]\label{def: fixed-point relocator}
    Let $\Gamma\subseteq\mathbb{R}_{++}$ be a nonempty set and let $(T_{\gamma})_{\gamma\in\Gamma}$ be a family of operators from $\mathcal{H}$ to $\mathcal{H}$. The family of operators $(\mathcal{Q}_{\delta\leftarrow\gamma})_{\delta,\gamma\in\Gamma}$ on $\mathcal{H}$ is said to be \emph{fixed-point relocators} for $(T_{\gamma})_{\gamma\in\Gamma}$ with Lipschitz constants $(\mathcal{L_{\delta\leftarrow\gamma}})_{\delta,\gamma\in\Gamma}$ in $[1,\infty)$ if the following hold:
\begin{enumerate}[(i)]
    \item\label{def i} $(\forall \gamma,\delta\in\Gamma)$ $\mathcal{Q}_{\delta\leftarrow\gamma|\Fix T_\gamma}$ is a bijection from $\Fix T_\gamma$ to $\Fix T_\delta$.
    \item\label{def ii} $(\forall \gamma\in\Gamma)(\forall x\in\Fix T_{\gamma})$ the map $\delta\mapsto\mathcal{Q}_{\delta\leftarrow\gamma}x$ from $\Gamma$ to $\mathcal{H}$ is continuous.
    \item\label{def iii} $(\forall \gamma,\delta,\epsilon\in\Gamma)(\forall x\in\Fix T_{\gamma})$ $\mathcal{Q}_{\epsilon\leftarrow\delta}\mathcal{Q}_{\delta\leftarrow\gamma}x=\mathcal{Q_{\epsilon\leftarrow\gamma}}x$.
    \item\label{def iv} $(\forall \gamma,\delta\in\Gamma)$ $\mathcal{Q}_{\delta\leftarrow\gamma}$ is $\mathcal{L_{\delta\leftarrow\gamma}}$-Lipschitz continuous.
\end{enumerate}
\end{definition}
\begin{remark}\label{remark: identity relocator}
    For any $\epsilon\in\Gamma$, the relocator $\mathcal{Q}_{\epsilon\leftarrow\epsilon}$ coincides with the identity operator on $\Fix T_\epsilon$. To see this, let $\gamma,\epsilon\in\Gamma$ and $y\in\Fix T_{\epsilon}$. Then $x:=\mathcal{Q}_{\gamma\leftarrow\epsilon}y\in\Fix T_{\gamma}$ due to Definition~\ref{def: fixed-point relocator}(\ref{def i}). Applying Definition~\ref{def: fixed-point relocator}(\ref{def iii}) yields $\mathcal{Q}_{\epsilon\leftarrow\epsilon}y=\mathcal{Q}_{\epsilon\leftarrow\epsilon}\mathcal{Q}_{\epsilon\leftarrow\gamma}x=\mathcal{Q}_{\epsilon\leftarrow\gamma}x=y$, establishing the claim.
\end{remark}

 To set ideas, we provide the prototypical example of a fixed-point relocator presented in \cite{atenas2025relocated}, the fixed-point relocator of the Douglas--Rachford algorithmic operator \eqref{DR operator in pri}.

\begin{example}[Fixed-point relocator of Douglas--Rachford operator] \label{ex:FPR-DR}
    Let $\Gamma\subseteq\mathbb{R_{++}}$ be nonempty, $(T_\gamma)_{\gamma\in\Gamma}$ be the family of Douglas--Rachford algorithmic operators  defined in \eqref{DR operator in pri}. The operator given by
\begin{equation}\label{eq-NDR:relocator DR}
    \mathcal{Q}_{\delta\leftarrow\gamma}x=\frac{\delta}{\gamma}x+\left(1-\frac{\delta}{\gamma}\right)J_{\gamma A_1}x \quad\forall \delta, \gamma \in \mathbb{R}_{++}, \; \forall x \in \mathcal{H},
\end{equation} defines a fixed-point relocator for $T_\gamma$
with Lipschitz constant $\mathcal{L}_{\delta\leftarrow\gamma}=\max\{1,\frac{\delta}{\gamma}\}$ \cite[Lemma 4.7]{atenas2025relocated}.  In particular, $\mathcal{Q}_{\delta\leftarrow\gamma}^{-1} = \mathcal{Q}_{\gamma\leftarrow\delta}$ \cite[Lemma 3.1]{atenas2025relocated}, and $\mathcal{Q}_{\delta\leftarrow\gamma}(\Fix T_\gamma) = \Fix T_\delta$  \cite[Theorem 3.5]{atenas2025relocated}.
\end{example}
We recall the following convergence result for relocated fixed-point iterations. 

\begin{theorem}[Convergence of relocated fixed-point iterations {\cite[Theorem 4.5]{atenas2025relocated}}] \label{thm:Convergence of fixed-point relocator}
 Let $\Gamma \subseteq \mathbb{R}_{++}$ be nonempty,  $(T_{\gamma})_{\gamma\in\Gamma}$ be a family of $\alpha$-averaged nonexpansive operators on $\mathcal{H}$ with  $\Fix T_{\gamma} \neq\varnothing$ for all $\gamma\in\Gamma$, and $(\mathcal{Q}_{\delta\gets \gamma})_{\gamma,\delta \in \Gamma}$ be fixed-point relocators for $(T_{\gamma})_{\gamma \in \Gamma}$ with Lipschitz constants $(\mathcal{L}_{\delta\gets \gamma})_{\gamma,\delta \in \Gamma} $ in $[1,+\infty[$. 
 Suppose the map $\mathcal{H}\times \Gamma\to \mathcal{H}\colon (x,\gamma)\mapsto T_{\gamma}x$ is continuous, $(\gamma_n)_{n\in\mathbb{N}} \subseteq \Gamma$ converges $\gamma^*\in \Gamma$, and  $\sum_{n\in\mathbb{N}} (\mathcal{L}_{\gamma_{n+1}\gets \gamma_n} -1) < +\infty.$
 Let $x_0 \in \mathcal{H}$ and generate a sequence $(x_n)_{n\in\mathbb{N}}$ according to
\begin{equation} \label{e:FPR-iteration}
x_{n+1} := \mathcal{Q}_{\gamma_{n+1}\gets \gamma_n}T_{\gamma_n}x_n \quad \forall n\in\mathbb{N}.
\end{equation} Then $(x_n-T_{\gamma_n}x_n)_{n\in\mathbb{N}}$ converges strongly to zero, and $(x_n)_{n\in\mathbb{N}}$ converges weakly to some $x \in \Fix T_{\gamma^*}$.
\end{theorem}

Theorem~\ref{thm:Convergence of fixed-point relocator} establishes weak convergence of relocated fixed-point iterations but says nothing about the rate.

In what follows, we investigate the rate at which relocated fixed-point iterates converge assuming a linear error bound holds (see Definition~\ref{def: bounded linear regularity} below). To formalize this, recall that a sequence $(x_n)_{n\in\mathbb{N}}$ is said to converge \emph{$R$-linearly} to $x\in\mathcal{H}$ if there exist $C\in\mathbb{R_{+}}$ and $r\in(0,1)$ such that $\|x_n-x\|\leq Cr^n$ for all $n\in\mathbb{N}$.  Note that, in particular, this implies that $(x_n)_{n \in \mathbb{N}}$ converges strongly to $x$.

\section{Linear convergence of relocated fixed-point iterations}\label{section: 3}
In this section, we establish linear convergence of  relocated fixed-point iterations. % introduced in~\cite{atenas2025relocated}.
The key ingredient in our analysis is the following notion of bounded linear regularity for a family of operators.
\begin{definition}[Uniform linear regularity]\label{def: bounded linear regularity}
    Let $\Gamma\subseteq\mathbb{R}_{++}$ be a nonempty set and $(T_{\gamma})_{\gamma \in \Gamma}$ be a family of operators on $\mathcal{H}$. We say $(T_{\gamma})_{\gamma \in \Gamma}$ is \emph{uniformly boundedly linearly regular} if, for any nonempty bounded set $S \subseteq \mathcal{H}$, there exists $\kappa >0$ such that
    \begin{equation}\label{eq:ulr}
    \dist(x, \Fix T_\gamma)\leq \kappa \|x - T_\gamma x\|\quad\forall x \in S,\forall \gamma\in\Gamma.
\end{equation}
When \eqref{eq:ulr} holds with $S=\mathcal{H}$, we say $(T_{\gamma})_{\gamma \in \Gamma}$ is \emph{uniformly linearly regular} (with constant $\kappa$). 
\end{definition}

\begin{remark}
Some comments on Definition~\ref{def: bounded linear regularity} are in order. When $\Gamma=\{\gamma^*\}$ is a singleton, uniform (bounded) linear regularity for the family $(T_\gamma)_{\gamma\in\Gamma}$ reduces to \emph{(bounded) linear regularity} of operator $T_{\gamma^*}$ (see \cite[Definition~2.1]{bauschke2015linear}), which is closely related to \emph{metric subregularity} of $(\Id-T)$. In the constant stepsize setting, metric subregularity of $(\Id-T)$ implies the condition of coercivity used to establish linear convergence of fixed-point iterations; see \cite[eq. (3.1), Lemma~3.1]{hesse2013nonconvex} and \cite[Theorem~1.6]{aspelmeier2016local}. When $\Gamma$ is finite, (bounded) linear regularity is one of the key ingredients for analyzing convergence of the ``cyclic algorithm'' for finding a common fixed-point assuming one exists \cite[Corollary~6.2]{bauschke2015linear} (\emph{i.e.,} an element of $\cap_{\gamma\in\Gamma}\Fix T_{\gamma}$). See also \cite{borwein2017convergence} for discussion on H\"older versions of these properties. 

When $\Gamma$ is countable, \cite[Definition 4.1]{cegielski2018regular} proposes a regularity notion for sequences of operators. Precisely, given a nonempty closed convex set $C \subseteq \mathcal{H}$, a sequence of operators $(T_n)_{n\in\mathbb{N}}$ is said to be \emph{boundedly linearly $C$-regular} if, for any bounded set $S \subseteq \mathcal{H}$, there exists $\kappa >0$ such that 
$$\dist(x, C)\leq \kappa \|x - T_n x\|\quad\forall x \in S,\,\forall n \in \mathbb{N}.$$
As  demonstrated in \cite{cegielski2018regular}, this regularity condition is crucial for analyzing the rate of convergence of fixed-point iterations of the type $x_{n+1} = T_n x_n$ for all $n \in \mathbb{N}$ for finding a point in $C:=\bigcap_{n \in \mathbb{N}} \Fix T_n$ when one exists (see, e.g., \cite[Theorem 6.1]{cegielski2018regular}). Although nonemptiness of $C$ is a reasonable assumption in feasibility problems (for instance),  as noted in \cite[Remark 4.13]{atenas2025relocated}, it need not hold in the setting of relocated fixed-point iterations, even for simple structured optimization problems.
\end{remark}

In the following, Lemma~\ref{lemma: contraction implies regularity} shows that a family of contractions satisfies the condition of uniform (bounded) linear regularity. The result extends \cite[Lemma~5.5]{van2023bounding}, which shows that a single contraction mapping satisfies linear regularity.
\begin{lemma}\label{lemma: contraction implies regularity}
    Let $\Gamma\subseteq\mathbb{R}_{++}$ be a nonempty set and $(T_\gamma)_{\gamma\in\Gamma}$ be a family of $\beta$-contractions on $\mathcal{H}$. Then $(T_\gamma)_{\gamma\in\Gamma}$ is uniformly linearly regular with constant $\kappa=1/(1-\beta)$.
\end{lemma}
\begin{proof} 
Let $\gamma \in \Gamma$ be arbitrary. Since $T_\gamma$ is a contraction, it has a unique fixed-point  \cite[Theorem~1.50]{bauschke2017convex}. That is, $\Fix T_\gamma=\{x_\gamma\}$ for some $x_\gamma\in\mathcal{H}$. For any $x\in \mathcal{H}$, the triangle inequality and the fact that $T_\gamma$ is a $\beta$-contraction yields
\begin{equation}\label{eq-NDR: eq contraction implies regualarity}
        \|x-x_\gamma\|=\|x-T_\gamma x_\gamma\|
        \leq \|x-T_{\gamma}x\|+\|T_\gamma x-T_\gamma x_\gamma \|
        \leq \|x-T_\gamma x\|+\beta\|x-x_\gamma\|.   
\end{equation}
Rearranging \eqref{eq-NDR: eq contraction implies regualarity} yields
$$\dist(x,\Fix T_{\gamma})=\|x-x_\gamma\|\leq\frac{1}{1-\beta}\|x-T_\gamma x\|,$$
which establishes the result.
\end{proof}

Now we consider the following assumption which plays a key role for establishing the linear convergence result,  by relating the rate of convergence of the stepsize sequence $(\gamma_n)_{n \in \mathbb{N}}$ to the relocated fixed-point iterations.
\begin{assumption}\label{assumption: Q is Lipschitz depending on gamma}
    Let $\Gamma\subseteq\mathbb{R}_{++}$ be a nonempty set,  $(T_\gamma)_{\gamma\in\Gamma}$ be a family of operators on $\mathcal{H}$, and $(\mathcal{Q}_{\delta\leftarrow\gamma})_{\gamma,\delta\in\Gamma}$ be fixed-point relocators for $(T_\gamma)_{\gamma\in\Gamma}$ on $\mathcal{H}$. 
    %For each bounded subset $S\subseteq\{(x,\gamma):x\in\Fix T_\gamma,\gamma\in\Gamma\}$,
    For each bounded subset $S\subseteq\cup_{\gamma\in\Gamma}\bigl(\Fix T_\gamma\times \{\gamma\}\bigr)$,    
    there exists a constant $L>0$ such that 
    $$ \|\mathcal{Q}_{\delta\leftarrow\gamma}x-\mathcal{Q}_{\gamma\leftarrow\gamma}x\|\leq L|\delta-\gamma|\quad\forall \delta\in\Gamma,\,\forall (x,\gamma)\in S. $$
\end{assumption}
In Section~\ref{s:DR}, we will show that the fixed-point relocators in Example~\ref{ex:FPR-DR} for the Douglas--Rachford operators satisfy Assumption~\ref{assumption: Q is Lipschitz depending on gamma} (see Lemma~\ref{lemma: DR relocator Lipschitz}).

The following establishes linear convergence of relocated fixed-point iterations initialized at a fixed-point of $T_{\gamma_0}$. This auxiliary result will be useful in the main result of this section, Theorem~\ref{thm: T satisfies linear error bound}.

\begin{lemma}\label{Lemma: c_n coverges linearly to c}
    Let $\Gamma\subseteq\mathbb{R}_{++}$ be a nonempty set, $(T_\gamma)_{\gamma\in\Gamma}$ be a family of operators on $\mathcal{H}$ with $\Fix T_{\gamma}\neq\emptyset$ for all $\gamma\in\Gamma$, and $(\mathcal{Q}_{\delta\leftarrow\gamma})_{\gamma,\delta\in\Gamma}$ be fixed-point relocators  for $(T_\gamma)_{\gamma\in\Gamma}$ on $\mathcal{H}$. Suppose Assumption~\ref{assumption: Q is Lipschitz depending on gamma} holds, and $(\gamma_n)_{n\in\mathbb{N}}\subseteq\Gamma$ is a sequence that converges $R$-linearly to $\gamma^*\in\Gamma$.
    Given $c_0\in\Fix T_{\gamma_0}$, generate a sequence $(c_n)_{n\in\mathbb{N}}$ according to
    \begin{equation}\label{eq-NDR:iteration for fixed-point relocator}
c_{n+1}:=\mathcal{Q}_{\gamma_{n+1}\leftarrow\gamma_n}c_n %=\mathcal{Q}_{\gamma_{n+1}\leftarrow\gamma_n}T_{\gamma_n}c_n
\quad \forall n\in\mathbb{N}.
    \end{equation}
   Then, $c_n\in\Fix T_{\gamma_n}$ for all $n\in\mathbb{N}$, and $(c_n)_{n\in\mathbb{N}}$ converges $R$-linearly to $\mathcal{Q}_{\gamma^*\leftarrow\gamma_0}c_0\in\Fix T_{\gamma^*}$.
\end{lemma}
\begin{proof}
     We first show that $c_n\in\Fix T_{\gamma_n}$ for all $n\in\mathbb{N}$. If $c_n\in\Fix T_{\gamma_n}$, then Definition~\ref{def: fixed-point relocator}(\ref{def i}) implies 
     $$c_{n+1}=\mathcal{Q}_{\gamma_{n+1}\leftarrow\gamma_{n}}c_n\subseteq \mathcal{Q}_{\gamma_{n+1}\leftarrow\gamma_{n}}(\Fix T_{\gamma_n})\subseteq \Fix T_{\gamma_{n+1}}.$$
     Since $c_0\in\Fix T_{\gamma_0}$ by assumption, the result follows by induction on $n$.
     
     We now turn our attention to convergence of $(c_n)_{n\in\mathbb{N}}$. Set $c:=\mathcal{Q}_{\gamma^*\leftarrow\gamma_0}c_0$. Then Definition~\ref{def: fixed-point relocator}(\ref{def i}) yields $c\in\Fix T_{\gamma^*}$. Combining Remark~\ref{remark: identity relocator} and Definition~\ref{def: fixed-point relocator}(\ref{def iii}) gives
      \begin{equation}\label{eq:c_0-c}
      c_0=Q_{\gamma_0\leftarrow\gamma_0}c_0=Q_{\gamma_0\leftarrow\gamma^*}Q_{\gamma^*\leftarrow\gamma_0}c_0=Q_{\gamma_0\leftarrow\gamma^*}c.
      \end{equation}
     Since $c_n\in\Fix T_n$ for all $n\in\mathbb{N}$, Definition~\ref{def: fixed-point relocator}(\ref{def iii}) together with \eqref{eq:c_0-c} gives
       $$ c_n = (Q_{\gamma_n\leftarrow\gamma_{n-1}}\dots Q_{\gamma_2\leftarrow\gamma_1}Q_{\gamma_1\leftarrow\gamma_0})c_0 = Q_{\gamma_n\gets\gamma_0}c_0=Q_{\gamma_n\gets\gamma_0}Q_{\gamma_0\gets\gamma^*}c=Q_{\gamma_n\gets\gamma^*}c\quad\forall n\in \mathbb{N}.$$
    
    By Assumption~\ref{assumption: Q is Lipschitz depending on gamma}, there exists $L>0$ such that    \begin{equation}\label{eq-NDR:c_n linearly converge to c}  
            \|c_n-c\|=\|\mathcal{Q}_{\gamma_n\leftarrow\gamma^*}c-Q_{\gamma^*\leftarrow\gamma^*}c\|
            \leq L|\gamma_n-\gamma^*|.
    \end{equation}
    Since $(\gamma_n)_{n\in\mathbb{N}}$ converges $R$-linearly to $\gamma^*$, \eqref{eq-NDR:c_n linearly converge to c} implies $(c_n)_{n\in\mathbb{N}}$ converges $R$-linearly to $c$. This completes the proof.
    \end{proof}

\begin{remark}
    We emphasize that the point $c_0\in \Fix T_{\gamma_0}$ in Lemma~\ref{Lemma: c_n coverges linearly to c} is only used to define the auxiliary sequence \eqref{eq-NDR:iteration for fixed-point relocator} of fixed points in the proof. It is not required for the implementation of the algorithm encoded by \eqref{e:FPR-iteration}.
\end{remark}

We are now ready to prove our main result (Theorem~\ref{thm: T satisfies linear error bound}) concerning linear convergence of relocated fixed-point iterations under uniform bounded linear regularity.

\begin{theorem}[Rate of convergence of relocated fixed-point iterations]\label{thm: T satisfies linear error bound}
    Let $\Gamma\subseteq\mathbb{R}_{++}$ be a nonempty set, $(T_\gamma)_{\gamma\in\Gamma}$ be a family of $\alpha$-averaged nonexpansive operators with $\Fix{T_{\gamma}}\neq\emptyset$ for all $\gamma\in\Gamma$, and $(\mathcal{Q}_{\delta\leftarrow\gamma})_{\gamma,\delta\in\Gamma}$ be fixed-point relocators for $(T_\gamma)_{\gamma\in\Gamma}$ on $\mathcal{H}$ with Lipschitz constant $(\mathcal{L_{\delta\leftarrow\gamma}})_{\gamma,\delta\in\Gamma}$ in $[1,+\infty)$. Suppose Assumption~\ref{assumption: Q is Lipschitz depending on gamma} holds, $\mathcal{H}\times \Gamma\to \mathcal{H}\colon (x,\gamma)\mapsto T_{\gamma}x$ is continuous, $(T_\gamma)_{\gamma\in\Gamma}$ is uniformly boundedly linearly  regular, $(\gamma_n)_{n\in\mathbb{N}}\subseteq\Gamma$ converges $R$-linearly to $\gamma^*\in\Gamma$ and $\sum_{n\in\mathbb{N}}(\mathcal{L}_{\gamma_{n+1}\leftarrow\gamma_n}-1)<+\infty$.
    Given $x_0\in\mathcal{H}$, generate a sequence $(x_n)_{n\in\mathbb{N}}$ according to
    \begin{equation}\label{eq-NDR:iteration for fixed-point relocator-error-bound}
x_{n+1}:=\mathcal{Q}_{\gamma_{n+1}\leftarrow\gamma_n}T_{\gamma_{n}}x_n \quad \forall n\in\mathbb{N}.
    \end{equation}
    Then the following assertions hold:
    \begin{enumerate}[(i)]
    \item \label{it i}$\big(\dist(x_n,\Fix T_{\gamma_n})\big)_{n \in \mathbb{N}}$ converges $R$-linearly to zero.
        %\item $(\|x_n-c_n\|)_{n\in\mathbb{N}}$ converges $R$-linearly to zero.
    \item\label{it ii:T satisfies linear error bound}$(x_n)_{n\in\mathbb{N}}$ and $(T_{\gamma_n}x_n)_{n\in\mathbb{N}}$ converge $R$-linearly to the same point in $\Fix T_{\gamma^*}$.
    \end{enumerate}
\end{theorem}
\begin{proof}
By Theorem~\ref{thm:Convergence of fixed-point relocator}, the sequence $(x_n)_{n\in\mathbb{N}}$ converges weakly to some $x\in\Fix T_{\gamma^*}$, hence, in particular, the sequence $(x_n)_{n\in\mathbb{N}}$ is bounded. Thus, since $(T_\gamma)_{\gamma\in\Gamma}$ is uniformly boundedly linearly regular, there exists a $\kappa>0$ such that 
\begin{equation}\label{eq:error bound in proof}
    \dist(x_n, \Fix T_{\gamma_n})\leq \kappa \|x_n - T_{\gamma_n} x_n\|\quad\forall n\in\mathbb{N}.
\end{equation}
For convenience, denote $\ell_n:=\mathcal{L}_{\gamma_{n+1}\leftarrow\gamma_n}$ and $p_n:=P_{\Fix T_{\gamma_n}}x_n$ for all $n \in \mathbb{N}$. The latter is well-defined as $\Fix T_{\gamma_n}$ is a nonempty, closed, convex set \cite[Corollary~4.24 \& Theorem~4.29]{bauschke2017convex}.
We also claim that $(p_n)_{n \in \mathbb{N}}$ is bounded. To see this, let $c_0\in\Fix T_{\gamma_0}$ and define $c_{n+1}:=\mathcal{Q}_{\gamma_{n+1}\leftarrow\gamma_n}c_n$. Then, by Lemma~\ref{Lemma: c_n coverges linearly to c}, $c_n\in\Fix T_{\gamma_n}$ for all $n\in\mathbb{N}$ and $(c_n)_{n \in \mathbb{N}}$ converges (so, in particular, is bounded). Since $P_{\Fix T_{\gamma_n}}$ is nonexpansive, we have
$$ \|p_n-c_n\| = \|P_{\Fix T_{\gamma_n}}x_n-P_{\Fix T_{\gamma_n}}c_n\|\leq \|x_n-c_n\|,$$
from which boundedness of $(p_n)_{n \in \mathbb{N}}$ follows. 

 (\ref{it i})~Since $p_n\in\Fix_{T_{\gamma_{n}}}$, Definition~\ref{def: fixed-point relocator}(\ref{def i}) implies $\mathcal{Q}_{\gamma_{n+1}\leftarrow\gamma_n}p_n\in\Fix T_{\gamma_{n+1}}$. From Definition~\ref{def: fixed-point relocator}(\ref{def iv}), we also have that $\mathcal{Q}_{\gamma_{n+1}\leftarrow\gamma_n}$ is $\ell_n$-Lipschitz continuous. Hence, as $T_{\gamma_n}$ is $\alpha$-averaged and \eqref{eq:error bound in proof} holds, we have
    \begin{equation}\label{eq-NDR:distance}
    \begin{aligned}
    \dist^2(x_{n+1},\Fix T_{\gamma_{n+1}})&\leq\|x_{n+1}-\mathcal{Q}_{\gamma_{n+1}\leftarrow\gamma_n}p_n\|^2\\
    &=\|\mathcal{Q}_{\gamma_{n+1}\leftarrow\gamma_n}T_{\gamma_n}x_n-\mathcal{Q}_{\gamma_{n+1}\leftarrow\gamma_n}p_n\|^2\\
    &\leq\ell_n^2\|T_{\gamma_{n}}x_n-p_n\|^2\\
    &\leq \ell_n^2(\|x_n-p_n\|^2-\frac{1-\alpha}{\alpha}\|(\Id-T_{\gamma_n})x_n\|^2)\\
    &\leq \ell_n^2(\dist^2(x_n,\Fix T_{\gamma_n})-\frac{1-\alpha}{\alpha\kappa^2}\dist^2(x_n,\Fix T_{\gamma_n}))\\
    &=\ell_n^2(1-\frac{1-\alpha}{\alpha\kappa^2})\dist^2(x_n,\Fix T_{\gamma_n}).
    %&=\eta_n^2(1-\frac{1}{\kappa^2})\|x_n-c_n\|^2
    \end{aligned}
    \end{equation}
    As $\ell_n\to 1^+$ and $1-\frac{1-\alpha}{\alpha\kappa^2}<1$, there exists $n_0\in\mathbb{N}$ and $r\in(0,1)$ such that $\ell_n^2(1-\frac{1-\alpha}{\alpha\kappa^2})<r^2<1$ for all $n\geq n_0$. Hence \eqref{eq-NDR:distance} implies
     \begin{equation}\label{eq-NDR:12*}
         \dist(x_{n+1},\Fix T_{\gamma_{n+1}})\leq r\dist(x_{n},\Fix T_{\gamma_{n}})\leq \dots\leq r^{n+1-n_0}\dist(x_{n_0},\Fix T_{\gamma_{n_0}}),
     \end{equation}
     from which $R$-linear convergence to zero follows.

(\ref{it ii:T satisfies linear error bound})~Since $p_n\in\Fix T_{\gamma_n}$, Definition~\ref{def: fixed-point relocator}\eqref{def i} implies  $v_{k}:=\mathcal{Q}_{\gamma_{k}\leftarrow\gamma_n}p_n\in\Fix {T_{\gamma_{k}}}$ for any $k\in\mathbb{N}$. Then, by Remark~\ref{remark: identity relocator} and Definition~\ref{def: fixed-point relocator}\eqref{def iii}, we have $v_n=p_n$ and
     $$ v_{k+1} = \mathcal{Q}_{\gamma_{k+1}\leftarrow\gamma_n}p_n = \mathcal{Q}_{\gamma_{k+1}\leftarrow\gamma_k}\mathcal{Q}_{\gamma_{k}\leftarrow\gamma_n}p_n=\mathcal{Q}_{\gamma_{k+1}\leftarrow\gamma_k}T_{\gamma_k}v_k\quad\forall k\in\mathbb{N}. $$
   Since $\sum_{n\in\mathbb{N}}(\ell_n-1)<+\infty$, it follows that $M:=\prod_{n\in\mathbb{N}}\ell_n<+\infty$. Hence, by combining $\ell_n$-Lipschitz continuity of $\mathcal{Q}_{\gamma_{n+1}\leftarrow\gamma_n}$, $\alpha$-averagedness of $T_{\gamma_n}$ and $\ell_n \geq 1$, we obtain
     \begin{equation}\label{eq:v_n+m}
          \|x_{n+m}-v_{n+m}\| 
          \leq \left(\prod_{k=n}^{n+m-1}\ell_{k}\right)\|x_n-v_n\|\leq M\|x_n-p_n\|. 
    \end{equation}
   
    Since $(\gamma_n)_{n \in \mathbb{N}}$ and $(p_n)_{n \in \mathbb{N}}$ are bounded, Assumption~\ref{assumption: Q is Lipschitz depending on gamma} implies that there exists $L> 0$ such that
    \begin{equation}\label{eq-NDR: z-p}
        \begin{aligned}
            \|\mathcal{Q}_{\gamma_{n+m}\leftarrow\gamma_n}p_n-p_n\|&=\|\mathcal{Q}_{\gamma_{n+m}\leftarrow\gamma_n}p_n-\mathcal{Q}_{\gamma_n\leftarrow\gamma_n}p_n\|\\
            &\leq L|\gamma_{n+m}-\gamma_n|\\
            &\leq L(|\gamma_{n+m}-\gamma^*|+|\gamma^*-\gamma_n|).
        \end{aligned}
    \end{equation}
    Applying the triangle inequality followed by substituting \eqref{eq:v_n+m} and \eqref{eq-NDR: z-p} gives
        \begin{equation}\label{eq-NDR: x_n+m}
        \begin{aligned}
            \|x_{n+m}-x_n\|
            &\leq\|x_{n+m}-v_{n+m}\|+\|v_{n+m}-p_n\|+\|p_n-x_n\|\\
            &\leq M\|x_n-p_n\|+\|\mathcal{Q}_{\gamma_{n+m}\leftarrow\gamma_n}p_n-p_n\|+\|x_n-p_n\|\\
            &\leq (M+1)\dist(x_n,\Fix T_{\gamma_n})+L(|\gamma_{n+m}-\gamma^*|+|\gamma^*-\gamma_n|).
        \end{aligned}
    \end{equation}
    Since $\big(\dist(x_n,\Fix T_{\gamma_n})\big)_{n \in \mathbb{N}}$ converges $R$-linearly to zero and $(\gamma_n)_{n\in\mathbb{N}}$ converges $R$-linearly to $\gamma^*$, we conclude that $(\|x_{n+m}-x_n\|)_{m,n \in \mathbb{N}}$ converges $R$-linearly to zero for $m,n\rightarrow \infty$. In particular, this means that $(x_n)_{n\in\mathbb{N}}$ is a Cauchy sequence, and hence it converges strongly. By taking the limit as $m\to+\infty$ in \eqref{eq-NDR: x_n+m}, we therefore deduce that $(x_n)_{n\in\mathbb{N}}$ converges $R$-linearly to $x$.
    
    Finally, since $T_{\gamma_n}$ is $\alpha$-averaged and $p_n\in\Fix T_{\gamma_n} $, we have 
    $$\frac{1-\alpha}{\alpha}\|(\Id-T_{\gamma_n})x_n\|^2\leq \|x_n-p_n\|^2=\dist^2(x_n,\Fix T_{\gamma_n}).$$
    Together with the triangle inequality, this yields
     $$\|T_{\gamma_n}x_n-x\|\leq \|T_{\gamma_n}x_n-x_n\| + \|x_n-x\| \leq \sqrt{\frac{\alpha}{1-\alpha}}\dist(x_n,\Fix T_{\gamma_n}) + \|x_n-x\| , $$
     from which it follows that $(T_{\gamma_n}x_n)_{n\in\mathbb{N}}$ converges strongly and  $R$-linearly to $x$.
    The proof is now complete.
\end{proof}
\begin{remark} A closer look at the proof of Theorem~\ref{thm: T satisfies linear error bound}\eqref{it ii:T satisfies linear error bound} reveals that the sequence $(x_n)_{n\in\mathbb{N}}$ and $(T_{\gamma_n}x_n)_{n\in\mathbb{N}}$ still converge strongly (but not necessarily linearly) if $R$-linear convergence of $(\gamma_n)_{n\in\mathbb{N}}$ is replaced by convergence (without a rate).
\end{remark}

\begin{corollary}[Rate of convergence of relocated fixed-point iterations of contractions]\label{cor: linear convergence for contraction}
Let $\Gamma\subseteq\mathbb{R}_{++}$ be nonempty set, $(T_\gamma)_{\gamma\in\Gamma}$ be a family of $\beta$-contractions with $\Fix{T_{\gamma}}\neq\emptyset$ for all $\gamma\in\Gamma$, and $(\mathcal{Q}_{\delta\leftarrow\gamma})_{\gamma,\delta\in\Gamma}$ be fixed-point relocators for $(T_\gamma)_{\gamma\in\Gamma}$ on $\mathcal{H}$ with Lipschitz constant $(\mathcal{L_{\delta\leftarrow\gamma}})_{\gamma,\delta\in\Gamma}$ in $[1,+\infty)$. Suppose Assumption~\ref{assumption: Q is Lipschitz depending on gamma} holds, $\mathcal{H}\times \Gamma\to \mathcal{H}\colon (x,\gamma)\mapsto T_{\gamma}x$ is continuous, $(\gamma_n)_{n\in\mathbb{N}}\subseteq\Gamma$ converges $R$-linearly to $\gamma^*\in\Gamma$ and $\sum_{n\in\mathbb{N}}(\mathcal{L}_{\gamma_{n+1}\leftarrow\gamma_n}-1)<+\infty$.
    Given $x_0\in\mathcal{H}$, generate a sequence $(x_n)_{n\in\mathbb{N}}$ according to
    \begin{equation}\label{eq-NDR:iteration for fixed-point relocator-error-bound-cor}
x_{n+1}=\mathcal{Q}_{\gamma_{n+1}\leftarrow\gamma_n}T_{\gamma_{n}}x_n \quad \forall n\in\mathbb{N}.
    \end{equation}
    Then the following assertions hold:
    \begin{enumerate}[(i)]
    \item $(\dist(x_n,\Fix T_{\gamma_n}))_{n\in\mathbb{N}}$ converges $R$-linearly to zero.
        %\item $(\|x_n-c_n\|)_{n\in\mathbb{N}}$ converges $R$-linearly to zero.
        \item $(x_n)_{n\in\mathbb{N}}$ and $(T_{\gamma_n}x_n)_{n\in\mathbb{N}}$ converge $R$-linearly to the same point in $\Fix T_{\gamma^*}$.
    \end{enumerate}
\end{corollary}
\begin{proof}
Since $(T_\gamma)_{\gamma\in\Gamma}$ is a family of $\beta$-contractions, it is a family of $\alpha$-averaged operators with $\alpha:=(\beta+1)/2\in(1/2,1)$ by \cite[Proposition~4.38]{bauschke2017convex} and uniformly linearly regular by Lemma~\ref{lemma: contraction implies regularity}. The result thus follows from Theorem~\ref{thm: T satisfies linear error bound}.
\end{proof}

\begin{remark}
    When the sequence of stepsizes $(\gamma_n)_{n\in\mathbb{N}}$ does not converge $R$-linearly in Corollary~\ref{cor: linear convergence for contraction} (hence also in Theorem~\ref{thm: T satisfies linear error bound}), then linear convergence of the relocated fixed-point algorithm need not hold. To see this, let $\beta\in[0,1)$ and for $\gamma \in \mathbb{R}$, define the operator $T_{\gamma}:\mathbb{R}\rightarrow\mathbb{R}$ according to 
    $$T_\gamma x:=\gamma+\beta(x-\gamma).$$
    Then, for any $\gamma\in\mathbb{R}$, $\Fix T_{\gamma}=\{\gamma\}$ and 
    $$\|T_\gamma x-T_\gamma y\|=\beta\|x-y\|\quad\forall x,y\in\mathbb{R}.$$
    In particular, $T_{\gamma}$ is a $\beta$-contraction with fixed-point relocator $\mathcal{Q}_{\delta\leftarrow\gamma}=\delta$ for any $\delta\in\mathbb{R}$.
    
    Let $(\gamma_n)_{n\in\mathbb{N}}$ be a positive sequence which converges to some $\gamma^*>0$. Set $x_0=\gamma_0$ and consider the relocated fixed-point iteration given by
    $$x_{n+1}=\mathcal{Q}_{\gamma_{n+1}\leftarrow\gamma_n}T_{\gamma_n}x_n \quad\forall n\in\mathbb{N}.$$
    A direct computation shows that $x_{n}=\gamma_n$ for all $n\in\mathbb{N}$. Hence $(x_n)_{n\in\mathbb{N}}$ converges $R$-linearly if and only if $(\gamma_n)_{n\in\mathbb{N}}$ converges $R$-linearly. 
\end{remark}

\section{The relocated Douglas--Rachford algorithm} \label{s:DR}
In this section, we apply the results established in Section~\ref{section: 3} to derive linear convergence of the \emph{relocated Douglas--Rachford algorithm} \cite{atenas2025relocated}, which is also known as the \emph{non-stationary Douglas--Rachford algorithm} in \cite{lorenz2019non}. Given $A_1, A_2: \mathcal{H} \setto \mathcal{H}$  are maximally monotone operators, the (relocated) Douglas--Rachford algorithm aims to \begin{equation*}
    \text{find~} x\in \mathcal{H} \text{~such that~} 0 \in (A_1 + A_2)x.
\end{equation*} As outlined in the Introduction, the Douglas--Rachford algorithm generates a sequence $({x}_n)_{n \in \mathbb{N}}$ via ${x}_{n+1} = T_\gamma {x}_n$ for all $n \in \mathbb{N}$, where $T_\gamma$ is the operator in \eqref{DR operator in pri}. Under the usual maximal monotonicity assumptions, $(x_n)_{n\in\mathbb{N}}$ converges weakly to a point $x\in\Fix T_\gamma$ with $J_{\gamma A_1}x\in\zer(A_1+A_2)$.

The relocated Douglas--Rachford framework extends the classical setting by introducing relocators that allow the stepsize to vary while still guaranteeing convergence.
Given  $\gamma,\delta\in\mathbb{R_{++}}$, let $T_\gamma$ denote the Douglas--Rachford operator defined in \eqref{DR operator sec 1} and let $Q_{\delta\leftarrow\gamma}$ be the fixed-point relocator introduced in~\cite{atenas2025relocated} defined in \eqref{eq-NDR:relocator DR}.
Given an initial point $x_0\in\mathcal{H}$ and a sequence of stepsizes $(\gamma_n)_{n\in\mathbb{N}}\in\mathbb{R}_{++}$, the resulting relocated Douglas--Rachford sequence~\cite{atenas2025relocated}  $(x_n)_{n\in\mathbb{N}}$ is given by
\begin{equation}\label{eq-NDR: NDR iteration 0}
    \left\{\begin{aligned}
        %y_n&=J_{\gamma_n A_2}(2z_n-x_n)\\
        w_n&=T_{\gamma_n}x_n = x_n-J_{\gamma_nA_1}x_n+J_{\gamma_nA_2}(2J_{\gamma_nA_1}x_n-x_n)\\
        x_{n+1}&=\mathcal{Q}_{\gamma_{n+1}\leftarrow\gamma_n}w_n=\frac{\gamma_{n+1}}{\gamma_n}w_n+\left(1-\frac{\gamma_{n+1}}{\gamma_n}\right)J_{\gamma_n A_1}w_n
    \end{aligned}\right.\quad\forall n\in\mathbb{N}.
\end{equation} Denote $z_0=J_{\gamma_0A_1}x_0$. Then, as observed in \cite[Proposition~4.8]{atenas2025relocated}, \eqref{eq-NDR: NDR iteration 0} can be reformulated as Algorithm~\ref{a:reloc-DR}, performing one resolvent evaluation per operator in each iteration, unlike \eqref{eq-NDR: NDR iteration 0}.

\begin{algorithm}[!ht]
\caption{Relocated Douglas--Rachford algorithm for finding a zero of $A_1+A_2$. \label{a:reloc-DR}}
\SetKwInOut{Input}{Input}
\Input{Choose $x_0 \in \mathcal{H}$ and a stepsize $ \gamma_0 \in \mathbb{R}_{++}$.}

Set $z_0 = J_{\gamma_0 A_1}x_0$. 

\For{$n=0,1,2,\dots$}{
Step 1. Intermediate step. Compute \begin{equation} \label{DR:variable-stepsize-1}
    \left\{\begin{aligned}
    y_n & = J_{\gamma_nA_2}(2z_n - x_n) \\
    w_n & = x_n - z_n + y_n. \\
    \end{aligned}\right.
\end{equation}

Step 2. Next iterate. Define $\gamma_{n+1} \in \mathbb{R}_{++}$ and compute
\begin{equation} \label{DR:variable-stepsize-2}
    \left\{\begin{aligned}
    z_{n+1} & = J_{\gamma_n A_1}w_n \\
    x_{n+1} & = \dfrac{\gamma_{n+1}}{\gamma_n} w_n + \left( 1 - \dfrac{\gamma_{n+1}}{\gamma_n} \right) z_{n+1}. 
    \end{aligned}\right.
\end{equation}

}
\end{algorithm}

Observe that when $\gamma_n = \gamma \in \mathbb{R}_{++}$ for all $n \in \mathbb{N}$, Algorithm~\ref{a:reloc-DR} recovers the  traditional Douglas--Rachford algorithm with constant stepsize. The authors in \cite{lorenz2019non,atenas2025relocated} retrieve similar convergence guarantees to the traditional setting, while allowing variable stepsizes.

We start the analysis of rate of convergence of Algorithm~\ref{a:reloc-DR} with the following lemma, which shows that fixed-point relocators of the Douglas--Rachford operator in Example~\ref{ex:FPR-DR} satisfies Assumption~\ref{assumption: Q is Lipschitz depending on gamma}.
\begin{lemma}\label{lemma: DR relocator Lipschitz}
%Assumption~\ref{assumption: Q is Lipschitz depending on gamma} holds.
Let $\Gamma\subseteq\mathbb{R}_{++}$ be a nonempty closed interval, $T_\gamma$ be the Douglas--Rachford operator given in \eqref{DR operator in pri}, and $Q_{\delta\leftarrow\gamma}$ be the fixed-point relocator given by \eqref{eq-NDR:relocator DR}. Then, for any nonempty bounded set $S\subseteq\cup_{\gamma\in \Gamma}\bigl(\Fix T_{\gamma}\times\{\gamma\}\bigr)$, there exists $L\geq 0$ such that 
$$ \|\mathcal{Q}_{\delta\leftarrow\gamma}x-\mathcal{Q}_{\gamma\leftarrow\gamma}x\| \leq L|\delta-\gamma| \quad\forall\delta\in\Gamma,\forall (x,\gamma)\in S. $$
In other words, Assumption~\ref{assumption: Q is Lipschitz depending on gamma}  holds.
\end{lemma}
\begin{proof}
For any $(x,\gamma)\in S$ and $\delta\in\Gamma$, we have
\begin{equation}\label{eq-NDR: Lipschitz for DR}
        \begin{aligned}
            \|\mathcal{Q}_{\delta\leftarrow\gamma}x-\mathcal{Q}_{\gamma\leftarrow\gamma}x\|
            &=\left\|\frac{\delta}{\gamma}x+\left(1-\frac{\delta}{\gamma}\right)J_{\gamma A_1}x-x\right\|\\
            &=|\delta-\gamma|\frac{\|x-J_{\gamma A_1}x\|}{\gamma}.
        \end{aligned}
    \end{equation}
We claim that $\sup_{(x,\gamma)\in S}\frac{\|x-J_{\gamma A_1}x\|}{\gamma}<+\infty$. To see this, first note that the triangle inequality combined with nonexpansiveness of $J_{\gamma A_1}$ implies
 $$ \|x-J_{\gamma A_1}x\| \leq \|x-J_{\gamma A_1}0\|+\|J_{\gamma A_1}0-J_{\gamma A_1}x\| \leq \|J_{\gamma A_1}0\|+2\|x\|. $$
As $\gamma\mapsto J_{\gamma A_1}0$ is continuous in view of Remark~\ref{remark: Lipschitz continuity in resolvent} and $\Gamma$ is compact in $\mathbb{R}$, $\sup_{\gamma\in\Gamma}\|J_{\gamma A_1}0\|<+\infty$ and, since $S$ is bounded, $\sup_{(x,\gamma)\in S}\|x\|<+\infty$. Altogether, we have
$$ \sup_{(x,\gamma)\in S}\frac{\|x-J_{\gamma A_1}x\|}{\gamma} \leq \sup_{\gamma\in\Gamma}\frac{\|J_{\gamma A_1}0\|}{\gamma}+2\sup_{(x,\gamma)\in S}\frac{\|x\|}{\gamma} <+\infty, $$
which establishes the result. 
\end{proof}

In order to apply Theorem~\ref{thm: T satisfies linear error bound}, we first show that the family of Douglas--Rachford operators is uniformly linearly regular under the assumption of strong monotonicity relative  to a set (Definition~\ref{d:relative-to-a-set}) of the operators involved.
\begin{lemma}\label{lemma: error bound for DR} 
Let $\Gamma\subseteq\mathbb{R}_{++}$ be a nonempty closed interval and $A_1,A_2\colon\mathcal{H}\setto\mathcal{H}$ be maximally monotone and paramonotone. Suppose $A_1$ or $A_2$ is $\mu$-strongly monotone relative to $\mathcal{P}:=\zer(A_1+A_2)$, and $A_1^{-1}$ or $A_2^{-1}$ is $\rho$-strongly monotone relative to $\mathcal{D}:=\zer(A_1^{-1} - A_2^{-1} \circ (- \Id)).$  Then the Douglas--Rachford operator $T_\gamma$ given by \eqref{DR operator in pri} is linearly regular with $\kappa_\gamma=4(1+\max\{\frac{1}{\gamma\mu},\frac{\gamma}{\rho}\})$ for all $\gamma\in\Gamma$. Consequently, $(T_\gamma)_{\gamma\in\Gamma}$ is uniformly linearly regular.
\end{lemma}
\begin{proof}
%\color{blue}
Let $\gamma\in\Gamma$. In view of Propositions~\ref{p:T_gamma-decomp} and \ref{p:rel-to-a-set}, the assumptions imply $\gamma A_1$ or $\gamma A_2$ is $\gamma\mu$-strongly monotone relative to $\mathcal{P}_\gamma = \mathcal{P}$, and $(\gamma A_1)^{-1}$ or $(\gamma A_2)^{-1}$ is $\gamma^{-1}\rho$-strongly monotone relative to $\gamma\mathcal{D} = \mathcal{D}_\gamma.$ The proof that $T_\gamma$ is linearly regular with $\kappa_\gamma=4(1+\max\{\frac{1}{\gamma\mu},\frac{\gamma}{\rho}\})$ is analogous to \cite[Theorem~2]{pena2021linear} except that  Proposition~\ref{p:T_gamma-decomp} (which applies for paramonotone operators) is used in place of \cite[Proposition~1]{pena2021linear} (which applies for subdifferentials). Since $\kappa_\gamma$  is a continuous function of $\gamma$ on $\Gamma$, it attains its maximum on the compact set $\Gamma$. Hence $(T_\gamma)_{\gamma\in\Gamma}$ is uniformly linearly regular with $\kappa=\max_{\gamma\in\Gamma}\kappa_\gamma$.
\end{proof}

We are now ready to prove the main results (Corollaries~\ref{cor: DR error bound} and \ref{cor: relocated DR when contraction})  regarding linear convergence of the relocated Douglas--Rachford algorithm (cf. \cite[Corollary 4.11]{atenas2025relocated}).  
\begin{corollary}[Rate of convergence of relocated Douglas--Rachford]\label{cor: DR error bound}
    Let $\Gamma\subseteq\mathbb{R}_{++}$ be a nonempty closed interval and $A_1,A_2\colon\mathcal{H}\setto\mathcal{H}$ be maximally monotone and paramonotone such that $\mathcal{P}:=\zer(A_1+A_2)\neq\emptyset$. Suppose $A_1$ or $A_2$ is $\mu$-strongly monotone relative to $\mathcal{P}$, and $A_1^{-1}$ or $A_2^{-1}$ is $\rho$-strongly monotone relative to $\mathcal{D}:=\zer(A_1^{-1} - A_2^{-1} \circ (- \Id)).$ Let $(\gamma_n)_{n\in\mathbb{N}}\subseteq\Gamma$ converge $R$-linearly to $\gamma^*\in\Gamma$, and $T_\gamma$ be the Douglas--Rachford operator given in \eqref{DR operator in pri}. Given $x_0\in\mathcal{H}$, generate sequences $(x_n)_{n\in\mathbb{N}}$, $(w_n)_{n\in\mathbb{N}}, (z_n)_{n\in\mathbb{N}}$ and $(y_n)_{n\in\mathbb{N}}$ according to Algorithm~\ref{a:reloc-DR}. Then the following assertions hold:
\begin{enumerate}[(i)]
    \item \label{cor:fixed-point-convergence} $(x_n)_{n\in\mathbb{N}}$ and $(w_n)_{n\in\mathbb{N}}$ converge  $R$-linearly to the same point $x$ in $\Fix T_{\gamma^*}$.
    \item \label{cor:primal-convergence} $(z_n)_{n\in\mathbb{N}}$ and $(y_n)_{n\in\mathbb{N}}$ converge  $R$-linearly to the same point $z := J_{\gamma^*A_1}x$ in $\mathcal{P}$.
    
    \item \label{cor:dual-convergence} $\bigl(\frac{x_n - z_n}{\gamma_n}\bigr)_{n\in\mathbb{N}}$ and $\bigl(\frac{w_n - y_n}{\gamma_n}\bigr)_{n\in\mathbb{N}}$  converge $R$-linearly to the  point $g := \frac{x - z}{\gamma^*}$ in $\mathcal{D}$. 
\end{enumerate}
\end{corollary}

\begin{proof}
\begin{enumerate}[(\ref{cor:fixed-point-convergence})]
    \item First, % $T_\gamma$ is firmly nonexpansive for each $\gamma\in\Gamma$,
    the operators in $(T_\gamma)_{\gamma\in\Gamma}$ are $\frac{1}{2}$-averaged nonexpansive \cite[Remark~4.34(iii)]{bauschke2017convex}. Furthermore, by Lemma~\ref{lemma: error bound for DR}, $(T_{\gamma})_{\gamma\in\Gamma}$ is uniformly linearly regular.  Since $(\gamma_n)_{n\in\mathbb{N}}$ converges $R$-linearly to $\gamma^*$, then there exists $C\in\mathbb{R}_{+}$ and $r\in(0,1)$ such that, for all $n\in\mathbb{N}$, we have
\begin{equation} \label{eq:rate-diff}
    (\gamma_{n+1}-\gamma_n)_+ \leq |\gamma_{n+1}-\gamma_n|\leq|\gamma_{n+1}-\gamma^*|+|\gamma^*-\gamma_n|\leq Cr^{n+1}+Cr^n\leq C(1+r)r^n.
\end{equation}
In addition, since $(\gamma_n)_{n\in\mathbb{N}}$ contained in a closed interval $\Gamma\subset\mathbb{R}_{++}$, there exists $\gamma_{\rm low}>0$ such that $\gamma_n\geq\gamma_{\rm low}$ for all $n\in\mathbb{N}$. 
Thus,
$$\sum_{n=0}^\infty(\mathcal{L}_{\gamma_{n+1}\leftarrow\gamma_n}-1)=\sum_{n=0}^\infty\frac{(\gamma_{n+1}-\gamma_n)_+}{\gamma_n}\leq\frac{C(1+r)}{\gamma_{\rm low}} \sum_{n=0}^\infty r^n=\frac{C(1+r)}{\gamma_{\rm low}(1-r)}<+\infty.$$ 
     
    Moreover, by Lemma~\ref{lemma: DR relocator Lipschitz}, Assumption~\ref{assumption: Q is Lipschitz depending on gamma} holds and $(x,\gamma)\mapsto T_{\gamma}x$ is continuous by \cite[Corollary~3.6]{atenas2025relocated}. Since $\mathcal{P}\neq\emptyset$, it follows that $\Fix T_{\gamma}\neq\emptyset$ for all $\gamma\in\Gamma$ (see, for example, \cite[Lemma~4.2]{malitsky2023resolvent}).  The result then follows from Theorem~\ref{thm: T satisfies linear error bound}.
    
    \item[(\ref{cor:primal-convergence})] By \cite[Corollary~4.11]{atenas2025relocated}, $z=J_{\gamma^* A_{1}}x\in\mathcal{P}$.
    Set $S:=\{(x_n,\gamma_n):n\in\mathbb{N}\}\cup\{(x,\gamma^*)\}$. Since for all $n\in \mathbb{N}$, $x_n\in\mathcal{H}$, $\gamma_n\in\Gamma\subseteq\mathbb{R}_{++}$  and the limit satisfies $(x,\gamma^*)\in\mathcal{H}\times\mathbb{R}_{++}$, we have, $S\subset\mathcal{H}\times\mathbb{R}_{++}$. As $(x_n,\gamma_n)_{n\in\mathbb{N}}$ converges strongly to $(x,\gamma^*)$, then $S$ is compact \cite[Theorem 2.38]{aliprantis2006infinite}.
Hence by Remark~\ref{remark: Lipschitz continuity in resolvent}, there exists $L>0$ such that $(x,\gamma)\mapsto J_{\gamma A_1}x$ is $L$-Lipschitz on $S$. Since $J_{\gamma_n A_{1}}$ is nonexpansive, we then obtain
    \begin{equation}\label{eq-NDR:7}
    \begin{aligned}
        \|z_n-z\|&=\|J_{\gamma_nA_1}x_n-J_{\gamma^*A_1}x\|\\
        &\leq\|J_{\gamma_nA_1}x_n-J_{\gamma_nA_1}x\|+\|J_{\gamma_nA_1}x-J_{\gamma^*A_1}x\|\\
        &\leq\|x_n-x\|+L|\gamma_n-\gamma^*|.
        \end{aligned}
    \end{equation}
 Since  $(x_n)_{n\in\mathbb{N}}$ and $(\gamma_n)_{n\in\mathbb{N}}$ converge $R$-linearly to $x$ and $\gamma^*$, respectively, we conclude that $(z_n)_{n\in\mathbb{N}}$ converges $R$-linearly to $z$. Finally, since, $w_n=x_n-z_n+y_n$, we have
$$\|y_n-z\|=\|w_n-x_n+z_n-z\|\leq\|w_n-x\|+\|x_n-x\|+\|z_n-z\|,$$
from which it follows that $(y_n)_{n\in\mathbb{N}}$ also converges $R$-linearly to $z$.

\item[(\ref{cor:dual-convergence})] 

Since $(y,\gamma)\mapsto J_{\gamma A_i}y$ is continuous for $i=1,2$ by Remark~\ref{remark: Lipschitz continuity in resolvent},  $(x_n)_{n\in\mathbb{N}}$ and  $(w_n)_{n\in\mathbb{N}}$ converge strongly to $x$, and $(z_n)_{n\in\mathbb{N}}$ and $(y_n)_{n\in\mathbb{N}}$ converge strongly to $z$,  taking the limit on both sides as $n\to\infty$ in the first identity in \eqref{DR:variable-stepsize-1} and the first line in \eqref{DR:variable-stepsize-2} yields $z=J_{\gamma^*A_2}(2z-x)$ and $z=J_{\gamma^*A_1}x$. Denote $g:=(x-z)/\gamma^*$. Then, using the definition of the resolvent,  $-z\in -A_2^{-1}(-g)$ and $z\in A_1^{-1}g$ hold, from which it follows that $g\in\mathcal{D}$.
Furthermore, since $(\gamma_n)_{n \in \mathbb{N}}$ is convergent in $\Gamma$, there exists $\gamma_{\rm low} >0$ such that $\gamma_n \geq \gamma_{\rm low}$ for all $n \in \mathbb{N}$. Hence, 
\begin{equation*}
    \left\|\frac{x_n - z_n}{\gamma_n} - g\right\| \leq \frac{1}{\gamma_{\rm low}}\bigl(\|x_n -x\| + \|z_n - z\|\bigr) +  \frac{1}{\gamma_{\rm low}\gamma^*}|\gamma_n - \gamma^*| \|x - z\|, 
\end{equation*}
from which we deduce $R$-linear convergence of $\bigl(\frac{x_n - z_n}{\gamma_n}\bigr)_{n\in\mathbb{N}}$. In view of the second step in \eqref{DR:variable-stepsize-1}, we have $\left\|\frac{w_n - y_n}{\gamma_n} - g \right\| = \left\|\frac{x_n - z_n}{\gamma_n} - g \right\|$ from which $R$-linear convergence of  $\bigl(\frac{w_n - y_n}{\gamma_n}\bigr)_{n\in\mathbb{N}}$ follows. This completes the proof.%}
\end{enumerate}  
\end{proof}

\begin{remark}
A closer look at the proof of Corollary~\ref{cor: DR error bound} reveals that Attouch--Th\`era duality is only used in combination with strong monotonicity of $A_1$ and $A_2^{-1}$ (resp.\ strong monotonicity of $A_1^{-1}$ and $A_2$) to deduce uniform linear regularity of $(T_{\gamma})_{\gamma\in\Gamma}$. Thus, the conclusions of Corollary~\ref{cor: DR error bound} would remain unchanged after replacing the assumptions related to paramonotonicity and strong monotonicity with uniform linear regularity of $(T_{\gamma})_{\gamma\in\Gamma}$.
\end{remark}

When the Douglas--Rachford operators are contractions, we can apply Corollary~\ref{cor: linear convergence for contraction} to deduce linear rates of convergence. In the following result, we show under what conditions this assumption is satisfied.
\begin{lemma}\label{lemma:uniform contraction of DR}
     Suppose that $A_1:\mathcal{H}\rightarrow\mathcal{H}$ is monotone and $L$-Lipschitz, and $A_2:\mathcal{H}\setto\mathcal{H}$ is $\mu$-strongly maximally monotone. Let $\Gamma\subseteq\mathbb{R}_{++}$ be a nonempty closed interval, and $T_\gamma$ be the Douglas--Rachford operator given in \eqref{DR operator in pri}. Then there exists a $\beta\in[0,1)$ such that  $(T_{\gamma})_{\gamma\in\Gamma}$ is a family of $\beta$-contractions.
\end{lemma}
\begin{proof}
    Let $\gamma\in\Gamma$. Since $\gamma A_1$ is monotone and $(\gamma L)$-Lipschitz and $\gamma A_2$ is $(\gamma\mu)$-strongly monotone, \cite[Theorem~4.3]{moursi2019douglas} implies that $T_\gamma$ is a $\beta_\gamma$-contraction with 
    $$ \beta_\gamma := \frac{1}{2(1+\gamma\mu)}\left(\sqrt{2\gamma^2\mu^2+2\gamma\mu+1+2\left(1-\frac{1}{(1+\gamma L)^2}-\frac{1}{1+\gamma^2L^2}\right)\gamma\mu(1+\gamma\mu)}+1\right)$$ and $\beta_\gamma\in(0,1)$. Since $\beta_\gamma$ is a continuous function of $\gamma$, it attains its maximum value on the compact set $\Gamma$. Hence $(T_\gamma)_{\gamma\in\Gamma}$ is a family of $\beta$-contractions with $\beta:=\max_{\gamma\in\Gamma}\beta_\gamma<1$.
    
\end{proof}

\begin{corollary}[Rate of convergence of relocated Douglas--Rachford of contractions]\label{cor: relocated DR when contraction}
    Let $\Gamma\subseteq\mathbb{R}_{++}$ be a nonempty closed interval. Suppose $A_1\colon\mathcal{H}\to\mathcal{H}$ is maximally monotone and $L$-Lipschitz and $A_2\colon\mathcal{H}\setto\mathcal{H}$ is $\mu$-strongly maximally monotone with $\mathcal{P}:=\zer(A_1+A_2)\neq\emptyset$. Denote $\mathcal{D}:=\zer(A_1^{-1} - A_2^{-1} \circ (- \Id))$.  Let $(\gamma_n)_{n\in\mathbb{N}}\subseteq\Gamma$ converge $R$-linearly to $\gamma^*\in\Gamma$, and $T_\gamma$ be the Douglas--Rachford operator given in \eqref{DR operator sec 1}. Given $x_0\in\mathcal{H}$, generate sequences $(x_n)_{n\in\mathbb{N}}$, $(w_n)_{n\in\mathbb{N}}, (z_n)_{n\in\mathbb{N}}$ and $(y_n)_{n\in\mathbb{N}}$ according to Algorithm~\ref{a:reloc-DR}. Then the following assertions hold:
\begin{enumerate}[(i)]
    \item \label{cor:fixed-point-convergence for contraction} $(x_n)_{n\in\mathbb{N}}$ and $(w_n)_{n\in\mathbb{N}}$ converge  $R$-linearly to the same point $x$ in $\Fix T_{\gamma^*}$.
    \item \label{cor:primal-convergence for contraction} $(z_n)_{n\in\mathbb{N}}$ and $(y_n)_{n\in\mathbb{N}}$ converge  $R$-linearly to the same point $z := J_{\gamma^*A_1}x$ in $\mathcal{P}$.
    \item \label{cor:dual-convergence for contraction} $\bigl(\frac{x_n - z_n}{\gamma_n}\bigr)_{n\in\mathbb{N}}$ and $\bigl(\frac{w_n - y_n}{\gamma_n}\bigr)_{n\in\mathbb{N}}$  converge $R$-linearly to the  point $g := \frac{x - z}{\gamma^*}$ in $\mathcal{D}$.
\end{enumerate}
\end{corollary}
\begin{proof}
 By Lemma~\ref{lemma:uniform contraction of DR}, there exists $\beta\in[0,1)$ such that $(T_\gamma)_{\gamma\in\Gamma}$ is a family of $\beta$-contractions. The remainder of the proof is the same as Corollary~\ref{cor: DR error bound} except that Corollary~\ref{cor: linear convergence for contraction} is used in place of Theorem~\ref{thm: T satisfies linear error bound} to deduce $R$-linear convergence of $(x_n)_{n\in\mathbb{N}}$.
\end{proof}

\begin{remark}\label{contrast to NDR}
Some comments regarding the convergence results of the relocated Douglas--Rachford algorithm and the non-stationary Douglas--Rachford algorithm from \cite{liang2017local}  are in order. 

First, let us compare the assumptions on the stepsize sequences in the two settings. Consider the relaxation parameters in \cite{liang2017local} satisfying $\lambda_n=1$ for all $n \in \mathbb{N}$, and denote $(\hat{\gamma}_n)_{n \in \mathbb{N}}$ the sequence of stepsizes in \cite{liang2017local} converging to $\hat{\gamma}\in \mathbb{R}_{++}$.\begin{enumerate}[(i)]
    \item The summability condition $\sum_{n \in \mathbb{N}} (\mathcal{L}_{\gamma_{n+1}\gets\gamma_n}-1) < +\infty$ in Theorem~\ref{thm:Convergence of fixed-point relocator} (and Theorem~\ref{thm: T satisfies linear error bound}), in view of \cite[Corollary 4.11]{atenas2025relocated}, reduces to $\sum_{n \in \mathbb{N}} (\gamma_{n+1}-\gamma_n)_+ < +\infty$. In Theorem~\ref{thm:Convergence of fixed-point relocator}, the sequence $(\gamma_n)_{n \in \mathbb{N}}$ is convergent, then by \cite[Remark~4.10]{atenas2025relocated}, the summability condition is equivalent to $\sum_{n \in \mathbb{N}} |\gamma_{n+1}-\gamma_n| < +\infty$. Moreover, as argued in \cite[Remark~4.1(iii)]{liang2017local}, since $\inf_{n \in \mathbb{N}} \lambda_n >0$, then \cite[(H.3)]{liang2017local} reads $\sum_{n \in \mathbb{N}} |\hat{\gamma}_n-\hat{\gamma}| < +\infty$. These two conditions are, in general, not equivalent. Indeed, consider $\gamma^* >0$ and $\gamma_n = \gamma^* + \frac{1}{n}$ for all $n \in \mathbb{N}$. Then $\sum_{n \in \mathbb{N}} |\gamma_{n+1}-\gamma_n| = \sum_{n \in \mathbb{N}} \frac{1}{n(n+1)} < +\infty$, but $\sum_{n \in \mathbb{N}} |\gamma_n-\gamma^*| = \sum_{n \in \mathbb{N}} \frac{1}{n} = +\infty $.
    
    \item If convergence of the sequence of stepsizes happens at a $R$-linear rate, then the situation above changes. Let $C \in \mathbb{R}_+$ and $\eta\in(0,1)$ such that for all $n \in \mathbb{N}$, $|\gamma_n - \gamma^*| \leq C \eta^n$. In \cite[Theorem~6.1]{liang2017local}, this condition is stated as $|\hat{\gamma}_n - \hat{\gamma}| = \mathcal{O}(\eta^n)$. This assumption implies $\sum_{n \in \mathbb{N}}|\gamma_n - \gamma^*| < +\infty$, and together with \eqref{eq:rate-diff}, also yields $\sum_{n \in \mathbb{N}} |\gamma_{n+1}-\gamma_n| < +\infty$. In this manner, in the setting of Corollary~\ref{cor: relocated DR when contraction} and  \cite[Theorem~6.1]{liang2017local}, the assumption on the sequence of stepsizes are equivalent. 
\end{enumerate}

Second, let us comment on the problem assumptions in the two settings. For the non-stationary Douglas--Rachford algorithm, \cite[Theorem~6.1]{liang2017local} establishes local linear convergence by assuming that the structured convex optimization problem is \emph{partly smooth} relative to $C^2$-manifolds, whereas Corollary~\ref{cor: relocated DR when contraction} establishes global $R$-linear convergence of the relocated Douglas--Rachford algorithm under the assumptions of Lipschitz continuity and strong monotonicity.

\end{remark}

\begin{remark}\label{linear convergence of NDR}
    As discussed above, the relocated Douglas--Rachford algorithm is also known as the non-stationary Douglas--Rachford algorithm in \cite{lorenz2019non}. Hence, Corollaries~\ref{cor: DR error bound} and~\ref{cor: relocated DR when contraction} give conditions for $R$-linear convergence of the non-stationary Douglas--Rachford algorithm.  For further details on the relationship between the two algorithms, see \cite[Remark 4.12]{atenas2025relocated}.
\end{remark}

\section{The relocated Malitsky--Tam algorithm} \label{s:MT}
In this section, we use the result of Section~\ref{section: 3} to derive linear convergence of a relocated resolvent splitting algorithm introduced in \cite{atenas2025relocated}, which is based on the resolvent splitting algorithm introduced by Malitsky and Tam \cite{malitsky2023resolvent}. Let $N\geq 2$, and $A_1, \dots, A_N: \mathcal{H} \setto \mathcal{H}$ be maximally monotone operators. The method in \cite{malitsky2023resolvent} aims to \begin{equation*}
    \text{find~} x\in \mathcal{H}, \text{~such that~} 0 \in (A_1 + \dots + A_N)x,
\end{equation*} by performing separate resolvent evaluations on each $A_i$, $i=1, \dots, N$, in each iteration. The algorithmic operator is defined as follows. Given a parameter $\theta \in (0,1)$, and a stepsize $\gamma\in\mathbb{R}_{++}$, consider the operator $T_\gamma:\mathcal{H}^{N-1}\rightarrow\mathcal{H}^{N-1}$ given by\begin{equation}\label{operator of MT}
        T_\gamma \mathbf{x}=\mathbf{x}+\theta\begin{pmatrix}
        z^2-z^1\\\vdots\\z^N-z^{N-1}
    \end{pmatrix} \quad \forall \mathbf{x} = (x^1, \dots, x^{N-1})\in \mathcal{H}^{N-1},
    \end{equation}
    where $z^1,\dots,z^N\in\mathcal{H}$ are given by
    \begin{equation}\label{def of z of MT}
    \left\{\begin{aligned}
        z^1&=J_{\gamma A_1}x^1\\
        z^i&=J_{\gamma A_i}(z^{i-1}+x^i-x^{i-1})&&\text{for }i=2,\dots,N-1\\
        z^N&=J_{\gamma A_N}(z^1+z^{N-1}-x^{N-1}).
    \end{aligned}\right.
\end{equation} For all $\gamma \in \mathbb{R}_{++}$ \cite[Lemma 2]{malitsky2023resolvent}, \begin{equation}  \label{eq:MT-Fix-zer-1}
    \Fix T_\gamma \neq \varnothing \iff \zer(A_1 + \dots + A_N) \neq \varnothing,
\end{equation}and\begin{equation} \label{eq:MT-Fix-zer-2}
\begin{aligned}
    &\mathbf{x} = (x^1, \dots, x^{N-1}) \in \Fix T_\gamma \iff z \in \zer(A_1 + \dots + A_N),\\
    &\text{where~} \left\{\begin{aligned}
        &z = J_{\gamma A_1}x^1\\
        &z = J_{\gamma A_i}(x^i - x^{i-1} + z) \text{~for }i=2,\dots,N-1\\
        &z = J_{\gamma A_N}(2z - x^{N-1}).
    \end{aligned}\right.
\end{aligned}
\end{equation} Furthermore, the sequence $(\mathbf{x}_n)_{n \in \mathbb{N}}$, generated via $\mathbf{x}_{n+1} = T_\gamma \mathbf{x}_n$ for all $n \in \mathbb{N}$, converges weakly to a point $\mathbf{x} \in \Fix T_\gamma$, and the sequence $(\mathbf{z}_n)_{n \in \mathbb{N}}$, with $\mathbf{z}_n^i$ defined in \eqref{def of z of MT} for all $n \in \mathbb{N}$ and $i = 1, \dots, N-1$, converges weakly to a point $(z, \dots, z) \in \mathcal{H}^{N}$, where $z \in \zer(A_1 + \dots + A_N)$ \cite[Theorem 2]{malitsky2023resolvent}.

In \cite[Section 5]{atenas2025relocated}, the authors retrieve similar convergence guarantees for a variable stepsize version of the algorithm, described as follows. As shown in \cite[Proposition 5.5]{atenas2025relocated}, for $\gamma, \delta \in \mathbb{R}_{++}$, the operator $\mathcal{\Tilde{Q}}_{\delta\leftarrow\gamma}=(\mathcal{\Tilde{Q}}^1_{\delta\leftarrow\gamma},\dots,\mathcal{\Tilde{Q}}^{N-1}_{\delta\leftarrow\gamma}):\mathcal{H}^{N-1}\rightarrow\mathcal{H}^{N-1}$ defines a fixed-point relocator for $T_\gamma$ in \eqref{operator of MT}-\eqref{def of z of MT}, where, for $\mathbf{x} = (x^1, \dots, x^{N-1}) \in \mathcal{H}^{N-1}$,
\begin{equation}\label{relocator of MT}
    \left\{\begin{aligned}
        \mathcal{\Tilde{Q}}^1_{\delta\leftarrow\gamma}\mathbf{x}&=\frac{\delta}{\gamma}x^1+\left(1-\frac{\delta}{\gamma}\right)J_{\gamma A_1}x^1 &&\\
        \mathcal{\Tilde{Q}}^i_{\delta\leftarrow\gamma} \mathbf{x}&=\frac{\delta}{\gamma}(x^i-x^1)+\mathcal{\Tilde{Q}}^1_{\delta\leftarrow\gamma}\mathbf{x} &&\text{for~}i=2,\dots,N-1.
    \end{aligned}\right.
\end{equation}
Let $\mathbf{x}_0=(x_0^1,\dots,x_0^{N-1})\in\mathcal{H}^{N-1}$, $(\gamma_n)_{n\in\mathbb{N}}$ be a sequence in $\mathbb{R}_{++}$, and denote $z_0^1=J_{\gamma_0A_1}x_0^1$. As explained in \cite[Proposition~5.7]{atenas2025relocated}, the relocated version of the algorithm based on \eqref{operator of MT}-\eqref{def of z of MT} can be written as Algorithm~\ref{a:reloc-MT}, preserving the number of per-iteration resolvent evaluations of \cite[Algorithm 1]{malitsky2023resolvent}.

\begin{algorithm}[!ht]
\caption{Relocated resolvent splitting  for finding a zero of $\sum_{i=1}^NA_i$, $N\geq 2$. \label{a:reloc-MT}}
\SetKwInOut{Input}{Input}
\Input{Choose $ x_{0}^1, \dots, x_{0}^{N-1} \in \mathcal{H}$, a stepsize $\gamma_0 \in \mathbb{R}_{++}$ and a parameter $\theta \in]0,1[$.}

Define $z_{0}^1 = J_{\gamma_0 A_1}(x_{0}^1)$.

\For{$n=0,1,\dots$}{
Step 1. Intermediate step. Compute \begin{equation*} %\label{e:intermediate-N}
    \mathbf{w}_n = \mathbf{x}_n + \theta \begin{pmatrix}
        z_{n}^2 - z_{n}^1\\ \vdots \\ z_{n}^N - z_{n}^{N-1}
    \end{pmatrix}
\end{equation*} where $z_{n}^1, \dots, z_{n}^N \in \mathcal{H}$ are given by \begin{equation*} 
     \left\{\begin{aligned}
         z_{n}^i &= J_{\gamma_n A_i}(x_{n}^i + z_{n}^{i-1} - x_{n}^{i-1}) & \text{~for~} i = 2, \dots, N-1\\
         z_{n}^N &= J_{\gamma_n A_N}(z_{n}^1 + z_{n}^{N-1} - x_{n}^{N-1}).&
     \end{aligned}\right.
\end{equation*}

Step 2. Next iterate. Define $\gamma_{n+1} \in \mathbb{R}_{++}$ and compute
\begin{equation} \label{DR:variable-stepsize-N} 
    \begin{cases}
    z_{n+1}^1  = J_{\gamma_n A_1}w_{n}^1\\
    x_{n+1}^1  = \dfrac{\gamma_{n+1}}{\gamma_n} w_{n}^1 + \left( 1 - \dfrac{\gamma_{n+1}}{\gamma_n}\right) z_{n+1}^1
    \end{cases}
\end{equation}For $i = 2, \dots, N-1$, 
\begin{equation*} %\label{DR:variable-stepsize-N-b} 
    x_{n+1}^i  = \dfrac{\gamma_{n+1}}{\gamma_n} (w_{n}^i - w_{n}^1) + x_{n+1}^1.
\end{equation*}
}
\end{algorithm}

We begin our analysis of Algorithm~\ref{a:reloc-MT} by showing that 
the fixed-point relocator in \eqref{relocator of MT} satisfies Assumption~\ref{assumption: Q is Lipschitz depending on gamma}.
\begin{lemma}\label{lemma: MT relocator Lipschitz}

Let $\Gamma\subseteq\mathbb{R}_{++}$ be a nonempty closed interval, $T_\gamma$ be the operator given in \eqref{operator of MT}-\eqref{def of z of MT}, and for $\gamma,\delta \in \mathbb{R}_{++}$,  $Q_{\delta\leftarrow\gamma}$ be the fixed-point relocator given by \eqref{relocator of MT}. Then, for any nonempty bounded set $S\subseteq\cup_{\gamma\in \Gamma}\bigl(\Fix T_{\gamma}\times\{\gamma\}\bigr)$, there exists $L\geq 0$ such that 
$$ \|\mathcal{\tilde{Q}}_{\delta\leftarrow\gamma}\mathbf{x}-\mathcal{\tilde{Q}}_{\gamma\leftarrow\gamma}\mathbf{x}\| \leq L|\delta-\gamma| \quad\forall\delta\in\Gamma,\forall (\mathbf{x},\gamma)\in S. $$
In other words, Assumption~\ref{assumption: Q is Lipschitz depending on gamma}  holds.
\end{lemma}
\begin{proof}
For any $(\mathbf{x},\gamma)\in S$ and $\delta\in\Gamma$, similarly to \eqref{eq-NDR: Lipschitz for DR}, we have
\begin{equation}\label{eq-NDR: Lipschitz for MT-1}
        \begin{aligned}
            \|\mathcal{\tilde{Q}}^1_{\delta\leftarrow\gamma}\mathbf{x}-\mathcal{\tilde{Q}}^1_{\gamma\leftarrow\gamma}\mathbf{x}\|
            &=\left\|\frac{\delta}{\gamma}x_1+\left(1-\frac{\delta}{\gamma}\right)J_{\gamma A_1}x^1-x^1\right\|\\
            &=|\delta-\gamma|\frac{\|x^1-J_{\gamma A_1}x^1\|}{\gamma},
        \end{aligned}
    \end{equation}
and for $i=2\dots,N-1$, we have    
\begin{equation}\label{eq-NDR: Lipschitz for MT-i}
        \begin{aligned}
            \|\mathcal{\tilde{Q}}^i_{\delta\leftarrow\gamma}\mathbf{x}-\mathcal{\tilde{Q}}^i_{\gamma\leftarrow\gamma}\mathbf{x}\|
            &=\left\|\frac{\delta}{\gamma}(x^i-x^1)+\frac{\delta}{\gamma}x^1+\left(1-\frac{\delta}{\gamma}\right)J_{\gamma A_1}x^1-x^i\right\|\\
            &=|\delta-\gamma|\frac{\|x^i-J_{\gamma A_1}x^1\|}{\gamma}.
        \end{aligned}
    \end{equation}
To complete the proof, it suffices to show that $\sup_{(\mathbf{x},\gamma)\in S}\frac{\|x^i-J_{\gamma A_1}x^1\|}{\gamma}<+\infty$ for each $i=1,\dots,N-1$. This follows by using a similar arrangement to the one in Lemma~\ref{lemma: DR relocator Lipschitz}.
\end{proof}

The next lemma provides conditions under which the operator $T_\gamma$ in \eqref{operator of MT}-\eqref{def of z of MT} is a contraction.
\begin{lemma}\label{cor: MT operators satisfy contractions}
Let $\Gamma\subseteq\mathbb{R}_{++}$ be a nonempty closed interval and $T_\gamma$ be the operators given in \eqref{operator of MT}-\eqref{def of z of MT}. Suppose that one of the following holds:
\begin{enumerate}[(a)]
    \item $A_{1},\dots,A_{N-1}: \mathcal{H}\rightarrow \mathcal{H}$ are monotone and $L$-Lipschitz, and $A_{N}\colon \mathcal{H}\setto \mathcal{H}$ is maximally $\mu$-strongly monotone. 
    \item $A_{1},\dots,A_{N-1}: \mathcal{H}\rightarrow \mathcal{H}$ are maximally $\mu$-strongly monotone and $L$-Lipschitz, and $A_{N}\colon \mathcal{H}\setto \mathcal{H}$ is maximally monotone.
\end{enumerate}
Then there exists  $\beta\in[0,1)$ such that $(T_\gamma)_{\gamma\in\Gamma}$ is a family of $\beta$-contractions.
\end{lemma}
\begin{proof}
Since $\Gamma\subseteq\mathbb{R}_{++}$ is a nonempty closed interval, we may write $\Gamma=[\gamma_{\rm low},\gamma_{\rm high}]$ where $0<\gamma_{\rm low}<\gamma_{\rm high}$. Let $\gamma\in\Gamma$ be arbitrary. 
\begin{enumerate}[(a)]
\item Since $A_1,\dots,A_{N-1}$ are $L$-Lipschitz and $A_N$ is $\mu$-strongly monotone, $\gamma A_1,\dots,\gamma A_{N-1}$ are $(\gamma_{\rm high}L)$-Lipschitz and $\gamma A_N$ is $(\gamma_{\rm low}\mu)$-strongly monotone. \item  Since $A_1,\dots,A_{N-1}$ are $\mu$-strongly monotone and $L$-Lipschitz, $\gamma A_1,\dots,\gamma A_{N-1}$ are $(\gamma_{\rm low}\mu)$-strongly monotone and $(\gamma_{\rm high}L)$-Lipschitz. 
\end{enumerate}
In both cases, the strong monotonicity and Lipschitz continuity constants of $\gamma A_1,\dots,\gamma A_N$ can be expressed independently of $\gamma$. Hence the result follows from \cite[Lemma~3.4]{simi2025linear}.
\end{proof}

As mentioned above, the operator in \eqref{relocator of MT} is a fixed-point relocator of $T_\gamma$ in \eqref{operator of MT}-\eqref{def of z of MT}, and a Lipschitz constant is provided in \cite[Proposition 5.5]{atenas2025relocated}. In order to obtain the convergence results in this section, we require a Lipschitz constant in a more convenient form. We obtain this new Lipschitz constant by recalling the following identity, valid in any Hilbert space, 
\begin{equation} \label{eq:identity}
    \|\alpha u+(1-\alpha)v\|^2 = \alpha\|u\|^2+(1-\alpha)\|v\|^2-\alpha(1-\alpha)\|u-v\|^2\quad\forall u,v \in \mathcal{H}, \alpha \in \mathbb{R}.
\end{equation} 
\begin{lemma} \label{l:FPR-Lip}
    For $\delta,\gamma \in \mathbb{R}_{++}$, the operator $\mathcal{\Tilde{Q}}_{\delta\gets\gamma}: \mathcal{H}^{N-1} \to \mathcal{H}^{N-1}$ defined in \eqref{relocator of MT} is Lipschitz continuous with constant 
    \begin{equation*}
        \check{\mathcal{L}}_{\delta\gets\gamma} = \frac{\sqrt{\delta}}{\sqrt{\gamma}}+\frac{\sqrt{|\gamma-\delta|}}{\sqrt{\gamma}}\max\left\{\sqrt{N-1},  \sqrt{2N}\sqrt{\frac{\delta}{\gamma}}\right\}.
    \end{equation*}%}
\end{lemma}

\begin{proof}
    Let $\mathbf{x}=(x^1,\dots,x^{N-1}), \mathbf{w}=(w^1,\dots,w^{N-1}) \in \mathcal{H}^{N-1}$. Substituting the first line of \eqref{relocator of MT} into its second line yields 
\begin{equation*}\mathcal{\Tilde{Q}}^i_{\delta\leftarrow\gamma} \mathbf{x}=\frac{\delta}{\gamma}x^i+ \left( 1 - \frac{\delta}{\gamma} \right) J_{\gamma A_1}x_1 \text{~for~}i=1,\dots,N-1.
\end{equation*}
Thus, from \eqref{eq:identity}, it follows that
\begin{align*}    \|\mathcal{\Tilde{Q}}_{\delta\gets\gamma}\mathbf{x}-\mathcal{\Tilde{Q}}_{\delta\gets\gamma}\mathbf{w}\|^2
    &= \sum_{i=1}^{N-1} \|\mathcal{\Tilde{Q}}_{\delta\gets\gamma}^i\mathbf{x}-\mathcal{\Tilde{Q}}_{\delta\gets\gamma}^i\mathbf{w}\|^2 \\ & =
    \frac{\delta}{\gamma}\|\mathbf{x}-\mathbf{w}\|^2+\left( 1 - \frac{\delta}{\gamma} \right)(N-1)\|J_{\gamma A_1}x^1-J_{\gamma A_1}w^1\|^2 \\
    &\qquad - \frac{\delta}{\gamma}\left( 1 - \frac{\delta}{\gamma} \right)\sum_{i=1}^{N-1}\|(x^i-w^i)-(J_{\gamma A_1}x^1-J_{\gamma A_1}w^1)\|^2.
\end{align*} 
We distinguish two cases. If $\delta\leq\gamma$, then $1-\frac{\delta}{\gamma}\geq 0$ and so
\begin{align*}
    \|\mathcal{\Tilde{Q}}_{\delta\gets\gamma}\mathbf{x}-\mathcal{\Tilde{Q}}_{\delta\gets\gamma}\mathbf{w}\|^2
    &\leq \frac{\delta}{\gamma}\|\mathbf{x}-\mathbf{w}\|^2+\left( 1 - \frac{\delta}{\gamma} \right)(N-1)\|x^1-w^1\|^2\\
    &\leq \frac{\delta}{\gamma}\|\mathbf{x}-\mathbf{w}\|^2+\left( 1 - \frac{\delta}{\gamma} \right)(N-1)\|\mathbf{x}-\mathbf{w}\|^2\\
    &= \left(\frac{\delta}{\gamma}+\frac{(N-1)|\gamma-\delta|}{\gamma} \right)\|\mathbf{x}-\mathbf{w}\|^2.
\end{align*}
Else, suppose that $\delta>\gamma$. First note that, by nonexpansiveness of $J_{\gamma A_1}$, we have
$$ \|J_{\gamma A_1}x^1-J_{\gamma A_1}w^1\|^2 \leq \|x^1-w^1\|^2$$
and
\begin{align*}
    \|(x^i-w^i)-(J_{\gamma A_1}x^1-J_{\gamma A_1}w^1)\|^2
    & \leq 2\|x^i-w^i \| ^2 + 2\|J_{\gamma A_1}x^1-J_{\gamma A_1}w^1\|^2 \\
    &\leq 2\|x^i-w^i\|^2 + 2\|x^1-w^1\|^2.
\end{align*}
Thus, since $\frac{\delta}{\gamma}-1>0$, we have
\begin{align*}
    \|\mathcal{\Tilde{Q}}_{\delta\gets\gamma}\mathbf{x}-\mathcal{\Tilde{Q}}_{\delta\gets\gamma}\mathbf{w}\|^2
    &= \frac{\delta}{\gamma}\|\mathbf{x}-\mathbf{w}\|^2-\left(\frac{\delta}{\gamma}-1\right)(N-1)\|J_{\gamma A_1}x^1-J_{\gamma A_1}w^1\|^2 \\
    &\qquad + \frac{\delta}{\gamma}\left(\frac{\delta}{\gamma}-1\right)\sum_{i=1}^{N-1}\|(x^i-w^i)-(J_{\gamma A_1}x^1-J_{\gamma A_1}w^1)\|^2\\
    &\leq \frac{\delta}{\gamma}\|\mathbf{x}-\mathbf{w}\|^2 + \frac{\delta}{\gamma}\left(\frac{\delta}{\gamma}-1\right)\left(2\|\mathbf{x}-\mathbf{w}\|^2+2(N-1)\|\mathbf{x}-\mathbf{w}\|^2\right)\\
    &= \left(\frac{\delta}{\gamma}+2N\frac{\delta}{\gamma}\frac{|\delta-\gamma|}{\gamma}\right)\|\mathbf{x}-\mathbf{w}\|^2.
\end{align*}
Altogether, we get \begin{equation*}
        \|\mathcal{\Tilde{Q}}_{\delta\gets\gamma}\mathbf{x}-\mathcal{\Tilde{Q}}_{\delta\gets\gamma}\mathbf{w}\|\leq \hat{\mathcal{L}}_{\delta\gets\gamma}\|\mathbf{x}-\mathbf{w}\|,
    \end{equation*} where \begin{equation*}
        \hat{\mathcal{L}}_{\delta\gets\gamma}=\max\left\{ \sqrt{\frac{\delta}{\gamma} + \frac{(N-1)|\gamma-\delta|}{\gamma}}, \sqrt{\frac{\delta}{\gamma} + 2N \frac{\delta}{\gamma}\frac{|\gamma-\delta|}{\gamma}}\right\}.
    \end{equation*} A direct calculation using the fact that for all $\alpha,\beta \in \mathbb{R}_{++}$, $\sqrt{\alpha + \beta} \leq \sqrt{\alpha} + \sqrt{\beta} $ yields $\hat{\mathcal{L}}_{\delta\gets\gamma} \leq \check{\mathcal{L}}_{\delta\gets\gamma}$, from where we conclude the result.
\end{proof}

\begin{remark}
    It is clear from the proof in Lemma~\ref{l:FPR-Lip} above that $\hat{\mathcal{L}}_{\delta\gets\gamma}$ provides a tighter bound for the Lipschitz constant of the fixed-point relocator. Nevertheless, as we shall see in the next result, the Lipschitz constant $\check{\mathcal{L}}_{\delta\gets\gamma}$ provides a simpler convergence analysis in our setting.
\end{remark}

Now we are ready to prove the main result regarding linear convergence of~Algorithm~\ref{a:reloc-MT}.

\begin{corollary}[Rate of convergence of relocated Malitsky--Tam algorithm]
Let $\Gamma\subseteq\mathbb{R}_{++}$ be a nonempty closed interval. Suppose one of the following holds:
 \begin{enumerate}[(a)]
     \item \label{condition 1} $A_1,\dots,A_{N-1}:\mathcal{H}\rightarrow\mathcal{H}$ are monotone and $L$-Lipschitz, and $A_N:\mathcal{H}\setto\mathcal{H}$ is maximally $\mu$-strongly monotone.
     \item \label{condition 2}  $A_1,\dots,A_{N-1}:\mathcal{H}\rightarrow\mathcal{H}$ are maximally $\mu$-strongly monotone and $L$-Lipschitz, and $A_N:\mathcal{H}\setto\mathcal{H}$ is maximally monotone.
 \end{enumerate}
  Let $\zer(\sum_{i=1}^NA_i)\neq\emptyset$, $(\gamma_n)_{n\in\mathbb{N}}\subseteq\Gamma$ converge $R$-linearly to $\gamma^*\in\Gamma$, and $T_\gamma$ be the operator given in \eqref{operator of MT}-\eqref{def of z of MT}. Given $\mathbf{x}_0\in\mathcal{H}^{N-1}$, generate  sequences $(\mathbf{x}_n)_{n\in\mathbb{N}},(\mathbf{w}_n)_{n\in\mathbb{N}}\subset\mathcal{H}^{N-1}$ and $(\mathbf{z}_n)_{n\in\mathbb{N}}\subset\mathcal{H}^N$ according to Algorithm~\ref{a:reloc-MT}. Then the following assertions hold:
\begin{enumerate}[(i)]
    \item \label{cor:MT fixed-point-convergence for contraction} $(\mathbf{x}_n)_{n\in\mathbb{N}}$ and $(\mathbf{w}_n)_{n\in\mathbb{N}}$ converge  $R$-linearly to the same point 
    
    $\mathbf{x}=(x^1,\dots,x^{N-1})\in\Fix T_{\gamma^*}$.
    \item \label{cor:MT primal-convergence for contraction} $(\mathbf{z}_n)_{n\in\mathbb{N}}$ converge $R$-linearly to a point $(z,\dots,z)\in\mathcal{H}^N$ with $z := J_{\gamma^*A_1}x^1\in\zer(\sum_{i=1}^N A_i)$.
    \end{enumerate}
\end{corollary}
\begin{proof}
    By Lemma~\ref{cor: MT operators satisfy contractions}, there exists $\beta\in[0,1)$ such that $(T_\gamma)_{\gamma\in\Gamma}$ is a family of $\beta$-contractions. By Lemma~\ref{lemma: MT relocator Lipschitz}, Assumption~\ref{assumption: Q is Lipschitz depending on gamma} holds. Since $(x,\gamma)\mapsto J_{\gamma A_i}x$ is continuous on $\mathcal{H}\times\mathbb{R}_{++}$ by Remark~\ref{remark: Lipschitz continuity in resolvent}, it follows that $(x,\gamma)\mapsto T_\gamma x$ is continuous, by algebra and composition of continuous operators. 
    Next, since $(\gamma_n)_{n\in\mathbb{N}}$ converges $R$-linearly to $\gamma^*$, then there exist $C\in\mathbb{R}_{+}$, $r\in[0,1)$ and $\gamma_{\rm high}, \gamma_{\rm low} \in \mathbb{R}_{++}$, such that, for all $n\in\mathbb{N}$, $|\gamma_n - \gamma^*| \leq C r^n$, and $\gamma_{\rm low} \leq \gamma_n \leq \gamma_{\rm high}$. In view of Lemma~\ref{l:FPR-Lip}, for all $n \in \mathbb{N}$, the Lipschitz constant $\check{\mathcal{L}}_{\gamma_{n+1} \gets \gamma_n}$ of the fixed-point relocator $\mathcal{\Tilde{Q}}_{\gamma_{n+1}\gets\gamma_n}$ satisfies 
    \begin{align*}
        \check{\mathcal{L}}_{\gamma_{n+1} \gets \gamma_n} - 1  &=  \frac{\sqrt{\gamma_{n+1}}}{\sqrt{\gamma_n}}-1+\frac{\sqrt{|\gamma_{n+1}-\gamma_n|}}{\sqrt{\gamma_n}}\max\{\sqrt{N-1},\sqrt{2N}\sqrt{\frac{\gamma_{n+1}}{\gamma_n}}\}\\
        &\leq \frac{|\sqrt{\gamma_{n+1}}-\sqrt{\gamma_n}|}{\sqrt{\gamma_{\rm low}}} +\frac{\sqrt{|\gamma_{n+1}-\gamma_n|}}{\sqrt{\gamma_{\rm low}}}\max\{\sqrt{N-1},\sqrt{2N}\sqrt{\frac{\gamma_{\rm high}}{\gamma_{\rm low}}}\}
    \end{align*}
 Then, in view of \eqref{eq:rate-diff}, we have for all $n \in \mathbb{N}$,
 \begin{equation*}
        |\sqrt{\gamma_{n+1}} - \sqrt{\gamma_n}| = \frac{|{\gamma_{n+1}} - {\gamma_n}|}{\sqrt{\gamma_{n+1}} + \sqrt{\gamma_n}}
        \leq \frac{|{\gamma_{n+1}} - {\gamma_n}|}{2\sqrt{\gamma_{\rm low}}} \leq \frac{C(1+r)r^n}{2\sqrt{\gamma_{\rm low}}}.
    \end{equation*}
    Altogether, as $r \in (0,1)$,  there exists $M >0$ such that  \begin{equation*}
        \sum_{n=0}^\infty (\check{\mathcal{L}}_{\gamma_{n+1}\gets\gamma_n} -1) \leq M \sum_{n=0}^\infty\sqrt{r}^n = \frac{M}{1-\sqrt{r}} < +\infty. 
    \end{equation*}%}
\begin{enumerate}
    \item[(\ref{cor:MT fixed-point-convergence for contraction})] Since $\zer\bigl(\sum_{i=1}^NA_i\bigr)\neq\emptyset$, then $\Fix T_{\gamma}\neq\emptyset$ for all $\gamma\in\Gamma$ by \eqref{eq:MT-Fix-zer-1}. 
    From Corollary~\ref{cor: linear convergence for contraction}(ii), it follows that 
    $(x_n^i)_{n\in\mathbb{N}}$ and $(w_n^i)_{n\in\mathbb{N}}$ converge $R$-linearly to the same point $x^i$, for all $i=1,\dots, N-1$.
     \item[(\ref{cor:MT primal-convergence for contraction})] Denote \begin{equation} \label{eq:common-limit}
        z:=J_{\gamma^*A_1}(x^1)=J_{\gamma^*A_i}(z+x^i-x^{i-1})=J_{\gamma^*A_N}(2z-x^{N-1}) \text{~for~}i=2,\dots,N-1,
    \end{equation} 
    which is justified by \eqref{eq:MT-Fix-zer-2}. 
    
    Define 
    $$ S_i := \begin{cases}
    \{(w_n^1,\gamma_n):n\in\mathbb{N}\}\cup\{(x^1,\gamma^*)\} & i=1\\
    \{(z_{n}^{i-1}+x_n^i-x_n^{i-1},\gamma_n):n\in\mathbb{N}\}\cup\{(z+x^i-x^{i-1},\gamma^*):n\in\mathbb{N}\} &i=2,\dots,N-1\\
    \{(z_{n}^1+z_{n}^{N-1}-x^{N-1},\gamma_n)\}\cup\{(2z^1-x^{N-1},\gamma^*)\} & i=N.
              \end{cases} $$
    
    Since $(x_n)_{n\in\mathbb{N}}$ and $(w_n)_{n\in\mathbb{N}}$ converge weakly to $x^i$, and $(z_n)_{n\in\mathbb{N}}$ converges weakly to $z$, by \cite[Theorem~2.38]{aliprantis2006infinite}, each set $S_i$ is compact.  
    Hence, by Remark~\ref{remark: Lipschitz continuity in resolvent}, there exists $L_i>0$ such that $(y^i,\gamma)\mapsto J_{\gamma A_i}y^i$ is $L_i$-Lipschitz on $S_i$. Thus,  together with the first line in \eqref{DR:variable-stepsize-N}, \eqref{eq:common-limit}, the triangle inequality and nonexpansiveness of $J_{\gamma_nA_1}$, we obtain
    \begin{align*}
    \|z^1_{n+1}-z\| 
    &\leq \|J_{\gamma_nA_1}w^1_n-J_{\gamma_nA_1}x^1\|+\|J_{\gamma_nA_1}x^1-J_{\gamma^*A_1}x^1\|\\
    &\leq \|w^1_n-x^1\| + L_1|\gamma_n-\gamma^*|,
    \end{align*}
    from which $R$-linear convergence of $(z_n^1)_{n\in\mathbb{N}}$ to $z$ follows. Similarly, for $i=2,\dots,N-1$, we have
    \begin{align*}
    \|z^i_n-z\|
    &\leq \|J_{\gamma_nA_i}(z_n^{i-1}+x_{n}^i-x_n^{i-1})-J_{\gamma_nA_i}(z+x^i-x^{i-1})\|\\
    &\qquad  +\|J_{\gamma_nA_i}(z+x^i-x^{i-1})-J_{\gamma^*A_i}(z+x^i-x^{i-1})\|\\
    &\leq \|(z_n^{i-1}+x_{n}^i-x_n^{i-1})-(z+x^i-x^{i-1})\| + L_i|\gamma_n-\gamma^*| \\
    &\leq \|z_n^{i-1}-z\|+\|x_{n}^i-x^i\|+\|x_n^{i-1}-x^{i-1}\| + L_i|\gamma_n-\gamma^*|,
    \end{align*}
    from which $R$-linear convergence of $(z_n^i)_{n\in\mathbb{N}}$ to $z$ follows inductively. Finally, 
    \begin{align*}
        \|z^N_n-z\| & \leq \|J_{\gamma_nA_N}(z_n^1+z_n^{N-1}-x_n^{N-1})-J_{\gamma_nA_N}(2z-x^{N-1})\|\\
        &\qquad + \|J_{\gamma_nA_N}(2z-x^{N-1})-J_{\gamma^*A_N}(2z-x^{N-1})\| \\
        &\leq \|z_n^1-z\| + \|z_n^{N-1}-z\| + \|x_n^{N-1}-x^{N-1}\| + L_N|\gamma_n-\gamma^*|,
    \end{align*}
    from which $R$-linear convergence of $(z_n^N)_{n\in\mathbb{N}}$ to $z$ follows. This completes the proof.
\end{enumerate}

\end{proof}

\section{Conclusion}\label{s:conclusion}
In this work, we established linear rate of convergence of relocated fixed-point iterations, introduced in \cite{atenas2025relocated}, when the family of algorithmic operators $(T_{\gamma})_{\gamma\in\Gamma}$ satisfies a linear error bound. In particular, the framework applies in the setting where $(T_{\gamma})_{\gamma\in\Gamma}$ are contractions. We then applied our result to derive linear convergence of variable stepsize resolvent splittings for solving monotone inclusions namely, to the relocated Douglas--Rachford algorithm and relocated resolvent splitting algorithm due to Malitsky and Tam under the assumptions of Lipschitz continuity and strong monotonicity.

A direction for future research is to investigate whether the relocated framework can be extended to the more general setting of bounded H\"older regularity~\cite[Definition~2.7]{borwein2017convergence} which generalizes the notion of bounded linear regularity. Such an extension would allow for the analysis of operators that satisfy only H\"older-type error bounds and may lead to sublinear or linear convergence rates for relocated iterations. Another promising direction is to study adaptive stepsize rules for $\gamma_n$ that still guarantee linear convergence while improving practical performance, for instance, by shrinking the spectral radius of the algorithmic operator as a function of the stepsize (akin to \cite{lorenz2019non}). Also, it would be valuable to investigate linear rate of convergence by applying the relocated approach to the framework of resolvent splitting introduced in \cite{tam2024frugal}.

\section*{Acknowledgments}
The research of FA and MKT was supported in part by Australian Research Council grant DP230101749. The research of FAS was supported in part by Research Training Program Scholarship from the Australian Commonwealth Government and the University of Melbourne.

\section*{Data availability}Data sharing not applicable to this article as no data sets were generated or analyzed during the current study.

\section*{Conflict of Interest}
The authors declare that they have no relevant financial or non-financial interests related to the content of this article.

\section*{Contributions}All the authors contributed substantially in this work.

\bibliographystyle{splncs04}
\bibliography{reference}
\end{document}